\documentclass[12pt]{colt2017}
\usepackage{enumerate}

\usepackage{url, color, xcolor}
\usepackage{amsmath,amssymb,amsfonts,bbm,graphicx,xspace}
\usepackage{tcolorbox}
\definecolor{myblue}{HTML}{0000FF}
\definecolor{myred}{HTML}{FF0000}
\definecolor{myorange}{HTML}{FFA500}
\definecolor{myyellow}{HTML}{FFFF00}
\definecolor{mycyan}{HTML}{00FFFF}
\definecolor{mygreen}{HTML}{008000}
\definecolor{mybrown}{HTML}{A52A2A}
\renewcommand{\Pr}{\Prob}
\newcommand{\polylog}{\mathsf{polylog}}
\newcommand{\pop}{population recovery\xspace}

\newcommand{\Expect}{\mathbb{E}}

\newcommand{\ed}{\stackrel{\mathrm{def}}{=}}

\newcommand{\EE}{\mathbb{E}}

\newcommand{\cO}{O}
\newcommand{\cS}{\mathcal{S}}
\newcommand{\cP}{\mathcal{P}}

\newcommand{\eg}{e.g.\xspace}

\newcommand{\ignore}[1]{}%

\newcommand{\Th}{^\mathrm{th}}
\newcommand{\integers}{\mathbb{Z}}
\newcommand{\naturals}{\mathbb{N}}
\newcommand{\reals}{\mathbb{R}}
\newcommand{\Prob}{\mathbb{P}}

\newcommand{\Binom}{\mathrm{Bin}}

\newcommand{\pth}[1]{\left( #1 \right)}

\newcommand{\sth}[1]{\left\{ #1 \right\}}
\newcommand{\red}{\color{red}}
\newcommand{\blue}{\color{blue}}
\newcommand{\nb}[1]{{\sf\blue[#1]}}
\newcommand{\nbr}[1]{{\sf\red[#1]}}

\usepackage{prettyref}
\newrefformat{eq}{(\ref{#1})}
\newrefformat{thm}{Theorem~\ref{#1}}
\newrefformat{th}{Theorem~\ref{#1}}
\newrefformat{chap}{Chapter~\ref{#1}}
\newrefformat{sec}{Section~\ref{#1}}
\newrefformat{algo}{Algorithm~\ref{#1}}
\newrefformat{fig}{Fig.~\ref{#1}}
\newrefformat{tab}{Table~\ref{#1}}
\newrefformat{rmk}{Remark~\ref{#1}}
\newrefformat{clm}{Claim~\ref{#1}}
\newrefformat{def}{Definition~\ref{#1}}
\newrefformat{cor}{Corollary~\ref{#1}}
\newrefformat{lmm}{Lemma~\ref{#1}}
\newrefformat{prop}{Proposition~\ref{#1}}
\newrefformat{pr}{Proposition~\ref{#1}}
\newrefformat{app}{Appendix~\ref{#1}}
\newrefformat{apx}{Appendix~\ref{#1}}
\newrefformat{ex}{Example~\ref{#1}}
\newrefformat{exer}{Exercise~\ref{#1}}
\newrefformat{soln}{Solution~\ref{#1}}

\newcommand{\sfN}{{\sf N}}
\newcommand{\sfL}{{\sf L}}

\newcommand{\diff}{\text{d}}

\ignore{

}

\newtheorem{prop}[theorem]{Proposition}

\renewcommand{\hat}{\widehat}
\newcommand{\wt}{\widetilde}

\newcommand{\calM}{{\mathcal{M}}}

\newcommand{\calP}{{\mathcal{P}}}

\newcommand{\calS}{{\mathcal{S}}}

\newcommand{\calX}{{\mathcal{X}}}

\newcommand{\var}{\mathsf{var}}

\newcommand{\iprod}[2]{\left \langle #1, #2 \right\rangle}
\newcommand{\Iprod}[2]{\langle #1, #2 \rangle}
\newcommand{\indc}[1]{{\mathbf{1}_{\left\{{#1}\right\}}}}

\newcommand{\ones}{\mathbf{1}}

\newcommand{\TV}{\mathrm{TV}}

\def\eqdef{\triangleq}
\newcommand{\mreals}{\ensuremath{\mathbb{R}}}

\def\EE{\mathbb{E}\,}

\usepackage{tikz}
\usetikzlibrary{matrix,arrows,calc,shapes,backgrounds}
\usetikzlibrary{shapes.callouts,decorations.text} 

\usetikzlibrary{shapes.misc}

\tikzset{cross/.style={cross out, draw=black, minimum size=2*(#1-\pgflinewidth), inner sep=0pt, outer sep=0pt},
cross/.default={1pt}}

\tikzstyle{int}=[draw, fill=blue!20, minimum size=2em]
\tikzstyle{dot}=[circle, draw, fill=blue!20, minimum size=2em]
\tikzstyle{dotred}=[circle, draw, fill=red!20, minimum size=2em]
\tikzstyle{init} = [pin edge={to-,thin,black}]
\tikzstyle{initred} = [pin edge={to-,thin,red}]
\tikzstyle{plan}=[draw, fill=blue!20, minimum size=2em, text width=5em, rounded corners,align=center]
\tikzstyle{planwide}=[draw, fill=blue!20, minimum size=2em, text width=8em, rounded corners,align=center]

\usepackage{tikz-cd}

\usepackage{pgfplots}
\usepackage{pgffor}

\pgfplotsset{
    standard/.style={
        axis x line=bottom,
        axis y line=middle,
				clip=false,
    enlargelimits=upper,
        every axis x label/.style={at={(current axis.right of origin)},anchor=west},
        every axis y label/.style={at={(current axis.above origin)},anchor=south}
    }
}
\pgfkeys{/pgfplots/scale/.style={
  x post scale=#1,
  y post scale=#1,
  z post scale=#1}
}

\pgfplotsset{soldot/.style={color=blue,only marks,mark=*}}
\pgfplotsset{holdot/.style={color=blue,fill=white,only marks,mark=*}}

\newif\ifmapx
{\catcode`/=0 \catcode`\\=12/gdef/mkillslash\#1{#1}}
\edef\jobnametmp{\expandafter\string\csname horo_apx\endcsname}
\edef\jobnameapx{\expandafter\mkillslash\jobnametmp}
\edef\jobnameexpand{\jobname}
\ifx\jobnameexpand\jobnameapx
\mapxtrue
\else
\mapxfalse
\fi

\long\def\apxonly#1{\ifmapx{\color{blue}#1}\fi}

\title{Sample complexity of population recovery}

 \coltauthor{\Name{Yury Polyanskiy} \Email{yp@mit.edu}\\
 \addr Department of EECS, MIT, Cambridge, MA, USA
 \AND
 \Name{Ananda Theertha Suresh} \Email{theertha@google.com}\\
 \addr Google Research, New York, NY, USA
 \AND
 \Name{Yihong Wu} \Email{yihong.wu@yale.edu}\\
 \addr Department of Statistics and Data Science, Yale University, New Haven, CT, USA
 }

\begin{document}
\maketitle
\begin{abstract}
The problem of population recovery refers to estimating a distribution based on incomplete or corrupted samples.
Consider a random poll of sample size $n$ conducted on a population of individuals, where each pollee is asked to answer $d$ binary questions.
We consider one of the two polling impediments: 
\begin{itemize}
\item in lossy population recovery, a pollee may skip each question with probability $\epsilon$; 
\item in noisy population recovery, a pollee may lie on each question with probability $\epsilon$.
\end{itemize}
Given $n$ lossy or noisy samples, the goal is to estimate the probabilities of 
all $2^d$ binary vectors simultaneously within accuracy $\delta$ with high probability. 

This paper settles the sample complexity of population recovery.
For lossy model, the optimal sample complexity is $\tilde\Theta(\delta^{ -2\max\{\frac{\epsilon}{1-\epsilon},1\}})$,
improving the state of the art by Moitra and Saks in several ways: a lower bound is established, the upper bound is improved and the result depends at most on the logarithm of the dimension. Surprisingly, the sample complexity undergoes a phase transition from parametric to nonparametric rate when $\epsilon$ exceeds $1/2$. For noisy population recovery, the sharp sample complexity turns out to be more sensitive to dimension and scales as $\exp(\Theta(d^{1/3} \log^{2/3}(1/\delta)))$ except for the trivial cases of $\epsilon=0,1/2$ or $1$.

For both models, our estimators simply compute the empirical mean of a certain
function, which is found by pre-solving a linear program (LP). Curiously, the dual LP can be understood as Le Cam's method for lower-bounding the minimax risk, thus establishing the statistical optimality of the proposed estimators. The value of the LP is determined by complex-analytic methods.
\end{abstract}


\section{Introduction}
	\label{sec:intro}

\subsection{Formulation}
	\label{sec:formulation}

The problem of \emph{population recovery} refers to estimating an unknown distribution based on incomplete or corrupted samples.
Initially proposed by~\cite{DRWY12,WY12} in the context of learning DNFs with partial observations and further investigated in \cite{batman2013finding,MS13,LZ15,DST16}, this problem can also be viewed as a special instance of learning mixtures of discrete distributions in the framework of \cite{kearns1994learnability}.

The setting of \pop is the following:
Let $\cP_d$ be the set of all probability distributions over the hypercube
$\{0,1\}^d$. Let $P \in \cP_d$ be an unknown probability
distribution and $X \ed (X_1,X_2,\ldots X_d) \sim P$. Instead of
observing $X$, we observe its noisy version according to one of the two observation models:
\begin{itemize}
\item \textbf{Lossy population recovery}: For
  each $i$, $Y_i$ is obtained by passing $X_i$ independently through the 
	binary erasure channel with erasure probability $\epsilon$, 
	where 	$Y_i = X_i$ with probability $1-\epsilon$, and $Y_i = ?$,
  with probability $\epsilon$. 
  \item \textbf{Noisy population recovery}: 
    For each $i$, $Y_i$ is obtained by passing $X_i$ independently through the 
	binary symmetric channel with error probability $\epsilon$, 
	where		$Y_i = X_i$ with probability $1-\epsilon$, and $Y_i =
    1- X_i$, with probability $\epsilon$. 
\end{itemize}

Given independent noisy or lossy samples, the goal of \pop is to estimate the underlying distribution.  Specifically,
let $X^{(1)},\ldots,X^{(n)}$ be independently drawn from $P$ and we observe their noisy or lossy versions, denoted by
$Y^{(1)},\ldots,Y^{(n)}$, and aim to estimate the probabilities of all strings within $\delta$ simultaneously, i.e.,
$\|P-\hat P\|_\infty \leq \delta$ with high probability.\footnote{Equivalently, up to constant factors, we need to output
a set of strings $S \subset \{0,1\}^d$ and $\hat{P}_x$ for each $x \in S$, such that for all
$x \in S$, $|\hat{P}_x - P_x| \leq \delta$ and for all $x \notin S$, $P_x \leq \delta$. The point is that even when
dimension $d$ is large the list $S$ can be kept to a finite size of order ${1\over \delta}$.} In the absence of erasures
or errors, the problem is simply that of distribution estimation and the empirical distribution is near-optimal.  If the samples
are lossy or noisy, the distribution of each sample is a mixture of exponentially (in $d$) many product distributions
with mixing weights given by the input distribution $P$.  This problem is hence a special case of learning mixtures of discrete
product distributions introduced in \cite{kearns1994learnability} and more recently in
\cite{feldman2008learning,li2015learning}.

One of the key observations from \cite{DRWY12} is that both the sample and algorithmic complexity of estimating $P_x$ for all $x \in \{0,1\}^d$ is largely determined by those of estimating $P_x$ for a single $x$, which, without loss of generality, can be assumed to be the zero string.
This problem is referred to as \emph{individual recovery}.\footnote{This 
can be viewed as the combinatorial counterpart of estimating the density at a point, a well-studied problem in nonparametric statistics, where given $n$ iid samples drawn from a density $f$, the goal is to estimate $f(0)$ cf.~\cite{Tsybakov09}.}
Note that one can convert any estimator $\hat P_0$ to $\hat P_x$ by XOR-ing with $x$ the samples and applying the
estimator $\hat P_0$. 
However, naively applying union bound over all strings inflates both the time complexity and the error probability by a factor $2^d$, which is unacceptable.
The clever workaround in~\cite{DRWY12} is to leverage the special structure of Hamming space by recursively solving the
problem on lower-dimensional subspaces.
This argument is further explained in \prettyref{app:alg} along with an improved analysis. 
Specifically, given any estimator $\hat P_0$ using $n$ samples and time complexity $t$ such that $\Expect|\hat P_0 - P_0| \leq \delta$,
it can be converted to a distribution estimator $\hat P$ such that $\|\hat P-P\|_\infty \leq \delta$ with probability at least $1-\tau$ with 
$n \cdot \log \frac{d}{\delta \tau}$ samples and time complexity $t \cdot \frac{d}{\delta} \log \frac{d}{\delta \tau}$.
Therefore the problem of population recovery is equivalent, both statistically and algorithmically, to the problem of individual recovery.

To understand the statistical fundamental limit of this problem, we consider the \emph{minimax risk}, defined as:
\begin{equation}
R^*(n, d) = \inf_{\hat{P_0}}\sup_P \Expect_P[(\widehat{P_0}
  - P_0)^2].
\label{eq:minimaxp0}
\end{equation}
To be consistent with the existing literature, the main results in this paper are phrased in terms of \emph{sample complexity}:
\begin{equation}
n^*(\delta,d) = \min\{n: R^*(n, d) \leq \delta^2\},
\label{eq:nstar}
\end{equation}
with subscript $\sfL$ or $\sfN$ denoting the lossy or noisy observation model. Up to constant factors, $n^*(\delta,d)$ is also the minimal sample size such that $P_0$ can be estimated within an additive error of $\delta$ with probability, say, $1/3$.
The focus on this paper is to obtain sharp bounds on both $n^*_{\sfL}(d,\delta)$ and $n^*_{\sfN}(d,\delta)$ and computationally efficient estimators with provable optimality.
We make no assumption on the support size of the underlying distribution.

\subsection{Prior work}
	\label{sec:prior}

To review the existing results, we start with lossy \pop. Recall that $\epsilon$ is the erasure probability.
	A polynomial-time estimator is given in \cite{DRWY12} that succeeds  for $\epsilon \leq 0.635$, which was subsequently improved 
to 	$\epsilon \leq 1/\sqrt{2}$ by \cite{batman2013finding}.
The state of the art is \cite{MS13} who proposed a polynomial time algorithm that works for all $\epsilon < 1$, with sample complexity:
\begin{equation}
n^*_{\sfL}(d,\delta) \lesssim \pth{\frac{d}{\delta}}^{\frac{2}{1-\epsilon}\log\frac{2}{1-\epsilon}}.
\label{eq:ms13}
\end{equation}

It is worth noting that for the lossy model, most of the estimators for $P_0$ are of the following form:
\begin{equation}
\hat P_0 = \frac{1}{n} \sum_{i=1}^n g(w_i)
\label{eq:hatp}
\end{equation}
where $w_i=w(Y^{(i)})$ is the number of ones in the $i\Th$ sample $Y^{(i)}$.
Such an estimator is referred to as a \emph{linear estimator} since \prettyref{eq:hatp} can be equivalently written as a linear combination
$\hat P_0  = \frac{1}{n}\sum_{j=0}^d g(j) N_j$, where $N_j$ is the number of samples of Hamming weight $j$. Both \cite{MS13} and \cite{DST16} focus on efficient algorithms to construct the coefficient $g$.

The problem of noisy \pop was first studied by~\cite{WY12}. 
In contrast to the lossy model which imposes no assumption on the input, all provable results 
for the noisy model are obtained under a ``sparsity'' assumption, namely, the distribution $P$ has bounded support size $k$.
An algorithm for noisy recovery is proposed in \cite{WY12} with sample complexity $k^{\log k}$.
This was further improved by~\cite{LZ15} to $k^{\log \log k}$. Recently \cite{DST16} proposed a recovery algorithm with sample complexity that is polynomial in $(k/\delta)^{1/(1-2\epsilon)^4}$ and $d$, where $\epsilon$ is the probability to flip each bit.

Despite these exciting advances, several fundamental questions remain unanswered:
\begin{itemize}
\item Is it sufficient to restrict to the family of linear estimators of the form \prettyref{eq:hatp}? 
\item What is the optimal sample complexity of population recovery in terms of $\epsilon, d, \delta$? 
\item Estimators such as~\eqref{eq:hatp} only depend on the number of ones in the sample. For lossy \pop, in fact the number of ones and zeros in each sample together constitute a sufficient statistic (see \prettyref{rmk:type}) and a natural question arises: is the number of zeros in the samples almost uninformative for estimating $P_0$?
  \item For noisy \pop, is the assumption of bounded support size necessary for achieving polynomial (in $d$) sample complexity? 
  \end{itemize}
As summarized in the next two subsections, the main results in this paper settle all of these questions.  Specifically,
we find the optimal sample complexity and show that linear estimators~\eqref{eq:hatp} suffice to achieve it.
For lossy population recovery, the number of zeros indeed can be ignored. For the noisy model, without the bounded support assumption the sample complexity scales superpolynomially in the dimension.

\subsection{New results}
The main results of this paper provide sharp characterizations of the sample complexity for both lossy and noisy \pop, as well as computationally efficient estimators, 
shown optimal by minimax lower bounds.
We start from lossy recovery. 
The next result 
determines the optimal sample complexity up to a $\mathrm{polylog}$ factor, which turns out to be \emph{dimension-free}.
\begin{theorem}[Lossy population recovery]
\label{thm:bec}
There exist universal constants $c_1,c_2$ such that the following hold.
For the {low-erasure} regime of $\epsilon \leq 1/2$,
\begin{equation}
\frac{c_1}{\delta^2} \leq \sup_{d \in \naturals} n^*_{\mathsf{L}}(\delta,d) \leq \frac{1}{\delta^2}.
\label{eq:bec-low}
\end{equation}
For the {high-erasure} regime of $\epsilon > 1/2$, denoting $\delta_1 \triangleq {\delta \over 1-\epsilon}$, 
\begin{equation}
c_2 \Big(e^2 \delta_1 \log \frac{1}{\delta_1}\Big)^{ -\frac{2\epsilon}{1-\epsilon}}
\leq
\sup_{d \in \naturals} n^*_{\mathsf{L}}(\delta,d) \leq \delta^{
  -\frac{2\epsilon}{1-\epsilon}}\,,
\label{eq:bec-high}
\end{equation}
where the lower bound holds provided that $\delta_1 \leq e^{-1}$.
 Furthermore, \prettyref{eq:bec-low} and
\prettyref{eq:bec-high} also hold for any fixed dimension $d$ if 
$d \geq 1$ and $d \ge {64\over 1-\epsilon} \delta_1 ^{
-\frac{2\epsilon}{1-\epsilon}}$, respectively.
\end{theorem}

\begin{remark}[Elbow effects]
\label{rmk:elbow}
An important result in the statistics literature on nonparametric and high-dimensional functional estimation, the
\emph{elbow effect} refers to the phenomenon that there exists a critical regularity parameter (\eg, smoothness) below
which the rate of estimation is parametric $O(1/n)$ and above which, it becomes nonparametric (slower than $1/n$) in the
number of samples $n$. 
Taking a nonparametric view of the lossy population recovery problem by considering the input as binary sequence of infinite length, what we proved is the following characterization of the minimax estimation error:
for fixed $\epsilon$, as $n\to\infty$,
\begin{equation}
\inf_{\hat{P_0}}\sup_P \Expect_P[(\widehat{P_0} - P_0)^2] = 
\begin{cases}
 \Theta(n^{-1}) &  0\leq \epsilon\leq \frac{1}{2} \\
n^{ - \frac{1-\epsilon}{\epsilon}} \polylog(n) &  \frac{1}{2} < \epsilon < 1 \\
\end{cases}
\label{eq:minimax-rate}
\end{equation}
which exhibits the following elbow phenomenon:
\begin{itemize}	
	\item In the \emph{low-erasure} regime of $\epsilon \leq 1/2$, it is possible to achieve the parametric rate of $1/n$ in any dimension. Clearly this is the best one can hope for even when $d=1$, in which case one wants to estimate the bias of a coin based on $n$ samples but unfortunately $\epsilon$-fraction of the samples are lost.
\item In the \emph{high-erasure} regime of $\epsilon > 1/2$, the infinite dimensionality of the problem kicks in and the optimal rate of convergence becomes nonparametric and strictly slower than $1/n$.
\end{itemize}
Since elbow effects are known for estimating non-linear (\eg, quadratic) functionals 
(see, e.g., \cite{INK87,cai2005nonquadratic,FRW15}), 
	it is somewhat surprising that it is also manifested in lossy \pop where the goal is to estimate a simple linear functional, namely, $P_0$.

The lossy population recovery problem should also be contrasted with classical statistical problems in the presence of \emph{missing data}, such as low-rank matrix completion and principle component analysis. In the latter case, it has been shown in \cite{Lounici14} that 
if an $\epsilon$ faction of the coordinates of each sample are randomly erased, this effectively degrades the sample size from $n$ to $n \epsilon^2$. 
For lossy \pop, \prettyref{thm:bec} shows there is a phase transition in the effect of missing observations: when more
than half of the data are erased ($\epsilon > 1/2$), 
 the rate of convergence are penalized and the effective sample size drops from $n$ to $n^{(1-\epsilon)/\epsilon}$.
\end{remark}

Next we turn to the noisy population recovery. Perhaps surprisingly, the sample complexity is no longer dimension-free and in fact grows with the dimension $d$ super-polynomially. The next result determines it up to a constant in the exponent.

\begin{theorem}[Noisy population recovery]
\label{thm:bsc}
Let $\mu(\epsilon) \triangleq \frac{\epsilon(1-\epsilon)}{(1-2\epsilon)^2}$.
For all $d\geq 1$,
\[
\min\sth{
\exp(c_1 (1-2\epsilon)^2d), 
\exp\pth{c_1 \pth{d\mu(\epsilon) \log^2 \frac{1}{\delta}}^{1/3} }} \leq
  n^*_{\sfN}(\delta,d) \leq \exp\pth{ c_2 \pth{d\mu(\epsilon) \log^2 \frac{1}{\delta}}^{1/3} }.
\]
where $c_1,c_2$ are universal constants and the upper and lower bounds hold for all $\delta<1$ and all $\delta<1/3$, respectively.
\end{theorem}
\prettyref{thm:bsc} shows that to estimate $P_0$ within a constant accuracy $\delta$, the optimal sample size scales as 	
\[
 n^*_{\sfN}(\delta, d) = \exp\pth{
   \Theta\pth{\pth{\frac{\epsilon(1-\epsilon)}{(1-2\epsilon)^2}
     d \log^2\frac{1}{\delta}}^{1/3}} }.
\]
which is superpolynomial in the dimension.
This shows that the assumption made
by~\cite{kearns1994learnability,WY12} and subsequent work that the input distribution has bounded support size is in fact crucial for achieving polynomial sample complexity.
Indeed both algorithms in \cite{LZ15,DST16} are based on Fourier analysis of Boolean functions and exploit the sparsity of the distribution and they are not of the linear form \prettyref{eq:hatp}.

Finally, we mention that a subset of our results was
discovered independently by \cite{de2017sharp} using similar techniques.

\subsection{Technical contributions}

To describe our approach, we start from the constructive part. 
For both lossy and noisy \pop, we also focus on linear estimator \prettyref{eq:hatp}. 
Choosing the coefficient vector $g$ to minimize the worst-case mean squared error $\Expect (P_0-\hat P_0)^2$ leads to the following linear programming (LP):
\begin{equation}
\min_{g\in\reals^{d+1}}  \|\Phi^\top g-e_0\|_\infty + \frac{1}{\sqrt{n}} \|g\|_\infty,
\label{eq:LP-intro}
\end{equation}
where $e_0=(1,0,\ldots,0)^\top$ and $\Phi$ is a column stochastic matrix (probability transition kernel) that describes the conditional distribution of the output Hamming weight given the input weight.
Here the first and second term in \prettyref{eq:LP-intro} correspond to the bias and standard deviation respectively. 
In fact, for lossy recovery the (unique) unbiased estimator is a linear one corresponding to $g= (\Phi^\top)^{-1} e_0$. 
When the erasure probability $\epsilon\leq \frac{1}{2}$, this vector has bounded entries and hence the variance is $O(\frac{1}{n})$. This has already been noticed in \cite[Sec.~6.2]{DRWY12}. However, 
when $\epsilon>\frac{1}{2}$, the variance of the unbiased estimator is exponentially large and the LP \prettyref{eq:LP-intro} aims to achieve the best bias-variance tradeoff.

Surprisingly, we show that the value of the above LP also produces a lower bound that applies to \emph{any} estimator. This is done by relating the dual program of \prettyref{eq:LP-intro} to Le Cam's two-point method for proving minimax lower bound \cite{Lecam86}.  \prettyref{sec:general} formalizes this argument and introduces a general framework of characterizing the minimax risk of estimating linear functionals of discrete distributions by means of linear programming.

The bulk of the paper is devoted to evaluating the LP \prettyref{eq:LP-intro} for the lossy (\prettyref{sec:bec}) and noisy model (\prettyref{sec:bsc}). The common theme is to recast the dual LP in the function space on the complex domain, consider its $H^\infty$-relaxation, and use tools from complex analysis to bound the value.
Similar
 proof technique was previously employed in \cite{MS13} to upper-bound the value of the dual LP in order to establish the sample complexity upper bound in \prettyref{eq:ms13}.
Here we tighten the analysis to obtain the optimal exponent and dimension-free result. Furthermore, we show that the dual LP not only provides an upper bound on the estimation error, it also gives minimax lower bound via the general result in \prettyref{sec:general} together with a refinement based on Hellinger distance.
To show the impossibility results we need to bound the value of the dual LP from below, which we do by demonstrating
an explicit solution that plays the role of the least favorable prior in the minimax lower bound.

Initially, we were rather surprised to have non-trivial complex-analytic methods such Hadamard three-lines theorem
emerge in this purely statistical problem. Realizing to have rediscovered portions of~\cite{MS13}, we were further surprised
when a new estimator for the trace reconstruction problem\footnote{This problem, in short, aims to
reconstruct a binary string based on its independent observations through a memoryless deletion channel.} was recently proposed by \cite{nazarov2016trace} and \cite{de2016optimal}, whose design crucially relied on complex-analytic methods. In fact, one of the key steps
of~\cite{nazarov2016trace} relies on the results from~\cite{borwein1997littlewood}, which we also use in \prettyref{sec:bsc} to obtain
sharp bounds for noisy \pop.

Finally, in ~\prettyref{sec:smooth} we propose an alternative linear estimator, which requires a slightly worse sample
complexity \[ \delta^{ -2\max\{\frac{3\epsilon - 1}{1-\epsilon},1\}}, \] having the exponent off by a factor of at most
$2$ in the worst case of $\epsilon \to 1$. The advantage of this estimator is that it is explicit and does not require
pre-solving a $d$-dimensional LP. The estimator is obtained by applying the \emph{smoothing}
technique introduced in \cite{OSW16} that averages a sequence of randomly truncated unbiased estimators to achieve good
bias-variance tradeoff.

\subsection{Notations}
	\label{sec:notations}

	Let $\Binom(n,p)$ denote the binomial distribution with parameters $n$ and $p$.
	For a pair of probability distributions $P$ and $Q$, 
	let $P \otimes Q$ and $P*Q$ denote their product distribution and convolution, respectively.
	Let 
	$\TV(P,Q) = \frac{1}{2}\int |\diff P-\diff Q|$, $H^2(P,Q)=\int (\sqrt{\diff P}-\sqrt{\diff Q})^2$ and $H(P,Q) =
	\sqrt{H^2(P,Q)}$ denote the total variation, squared  and non-squared Hellinger distance, respectively.	
	Throughout the paper, for any $\epsilon \in (0,1)$, let $\bar{\epsilon} \triangleq 1-\epsilon$.	
	For a family of parametric distribution $\{P_\theta: \theta\in\Theta\}$ and a prior $\pi$ on $\Theta$, with a
	slight abuse of notation we write $\Expect_{\theta\sim \pi}[P_\theta]$ to denote
	the mixture distribution $\int \pi(\diff \theta) P_\theta$. 
	For a holomorphic function $f(z)$ on the unit disk $D\subset\mathbb{C}$ we let $[z^k]f(z)$ denote the $k$-th
	coefficient of its Maclaurin series.

\section{Estimating linear functionals and duality of Le Cam's method}
	\label{sec:general}
	
	
        In this section, we relate the general problem of estimating linear
        functionals of distributions to the solution of a 
				\emph{minimal total variation} linear program. The latter corresponds to
        the optimization of a pair of distributions in Le Cam's
        two-point method and thus produces minimax lower bounds.
				Surprisingly, we also show that dual linear program
        produces an estimator, thereby explaining the general tightness
        of the two-point method.
\begin{theorem}
\label{thm:dual}
        Let $\Theta$ and $\calX$ be finite sets and $\{P_\theta:\theta \in
        \Theta\}$ be a collection of distributions on $\calX$.  Let $\pi$
        be a distribution (prior) on $\Theta$ and let $X_1,\ldots,X_n$ be iid samples from the mixture distribution\footnote{In
	other words, for
        each $i=1,\ldots,n$, $\theta_i$ is drawn iid from $\pi$ and $X_i \sim P_{\theta_i}$ independently.} $\sum_{\theta\in\Theta}
	\pi(\theta) P_\theta$.
				To estimate a linear functional of $\pi$
\[
F(\pi)\triangleq \iprod{\pi}{h} = \Expect_{\theta\sim\pi}{h(\theta)},
\]
define the minimax quadratic risk
\begin{equation}
R_n \triangleq \inf_{\hat F} \sup_{\pi} \Expect_\pi[(\hat F(X_1,\ldots,X_n)-F(\pi))^2].
\label{eq:minimax-F}
\end{equation}
Without loss of generality\footnote{This is because $R_n$ is unchanged if $h$ is replaced by $h-c \ones$ for any constant $c$. Note, however,
that the value of~\eqref{eq:LP-G} does change.} we assume that there exists $\theta_+$ and $\theta_-$ such that
$h(\theta_-)\le 0 \le h(\theta_+)$.
Then we have
\begin{equation}
{1\over 64}\delta\pth{\frac{1}{n}}^2 \leq R_n \leq \delta\pth{\frac{1}{\sqrt{n}}}^2,
\label{eq:dual}
\end{equation}
where $\delta(t)$ is given by the following linear program
\begin{equation}
\delta(t) \triangleq \max_{\Delta \in\mreals^\Theta} \{\iprod{\Delta}{h}: \|\Delta\|_1\leq 1, \|\Phi \Delta\|_1\leq t\}\,,
\label{eq:LP-G}
\end{equation}
where 
$\Phi = (\Phi_{x,\theta})_{x\in \calX, \theta \in \Theta}$ is an $|\calX|\times |\Theta|$ column-stochastic matrix with
	$\Phi_{x,\theta} \eqdef P_\theta(x)$.
\end{theorem}
\begin{remark} 
\label{rmk:type}
Particularizing the above framework to the problem of population recovery, the subsequent two sections will be devoted to computing the LP value $\delta(t)$ for the lossy and noisy model, respectively. Below we specify the setting as a concrete example.
Since the probability of the zero vector is \emph{permutation-invariance}, it is sufficient\footnote{Indeed, for the worst-case formulation \prettyref{eq:minimaxp0}, the least favorable distribution $P$ is \emph{permutation-invariant}. Furthermore, if the input string $X$ has a permutation-invariant distribution, so is the distribution of the lossy or noisy observation $Y$, for which a sufficient statistic is its \emph{type}.}
 to consider permutation-invariant distributions and the relevant parameter is the Hamming weight of the input string, denoted by $\theta$, which is distributed according a distribution $\pi$ supported on the parameter space $\Theta=\{0,1,\ldots,d\}$. The goal, in turn, is to estimate $\iprod{\pi}{e_0} = \pi(0)$. Each sample is sufficiently summarized into their \emph{types}:
\begin{itemize}
	\item For lossy population recovery: $\calX=\{0,\ldots,d\}^2$ and $X=(U,V)$, where $U$ and $V$ denote the 
	number of ones and zeros, respectively, in the output. More precisely, we have 
	\begin{equation}
	X|\theta \sim \Binom(\theta,\bar\epsilon)\otimes \Binom(d-\theta,\bar\epsilon).
	\label{eq:bec}
	\end{equation}
	
	\item For noisy population recovery: $\calX=\{0,\ldots,d\}$ where $X$ is the number of ones in output, and so the model is given by
	\begin{equation}
	X|\theta \sim \Binom(\theta,\bar\epsilon) * \Binom(d-\theta,\epsilon)
	\label{eq:bsc}
	\end{equation}
\end{itemize}
\end{remark}

\begin{remark}
\label{rmk:hellinger-preview}
As will be evident from the proof, the square-root gap in the LP characterization \prettyref{eq:dual} is due the usage
of the bound $\TV(P^{\otimes n}, Q^{\otimes n}) \leq n\TV(P,Q)$, which is often loose. 
To close this gap, we resort to the Hellinger distance and for lossy population recovery we show in \prettyref{sec:bec} that
\[
R_n = \delta\pth{\frac{1}{\sqrt{n}}}^2 \polylog(n).
\]	
For noisy population recovery, it turns out that $\delta(t)$ is exponentially small in $\polylog(\frac{1}{t})$ and hence
\prettyref{eq:dual} is sufficiently tight (see \prettyref{sec:bsc}).
\end{remark}

\begin{proof}[Proof of \prettyref{thm:dual}]
To prove the right-hand inequality in~\eqref{eq:dual} we consider the following estimator
\begin{equation}
\hat F(X_1,\ldots,X_n) = \frac{1}{n}\sum_{i=1}^n g(X_i)
\label{eq:hatF}
\end{equation}
for some $g\in\reals^\calX$. 
 Then
\[
\Expect_\pi[\hat F(X_1,\ldots,X_n)] = \sum_{\theta \in\Theta} \pi(\theta) \sum_{x \in \calX} g(x)P_\theta(x) = \pi^\top \Phi^\top g = \Iprod{\pi}{\Phi^\top g}
\]
and the bias is
\[
|\Expect_\pi[\hat F] - F(\pi)| \leq \|\pi\|_1\|\Phi^\top g-h\|_\infty = \|\Phi^\top g-h\|_\infty.
\]
For variance,
\[
\var_\pi(\hat F) = \frac{1}{n} \var_\pi(g(X)) \leq \frac{1}{n} \|g\|_\infty^2.
\]
This shows
\begin{equation}\label{eq:dy1}
	\sqrt{R_n} \le \sup_\pi \left(\EE_\pi [|\hat F - F(\pi)|^2]\right)^{1\over 2} \le \|\Phi^\top g-h\|_\infty +
\frac{1}{\sqrt{n}} \|g\|_\infty\,.
\end{equation}
We now optimize the right-hand side of~\eqref{eq:dy1} over all $g$ as follows:
	\begin{align}	
		& ~  \min_{g\in \mreals^\calX} \|\Phi^\top g - h\|_{\infty} + \frac{1}{\sqrt{n}} \|g\|_\infty
			\label{eq:dy2a}\\
	= & ~ \min_{g} \max_{\|y\|_1\leq 1,\|z\|_1\leq 1} \Iprod{y}{\Phi^\top g - h} + \frac{1}{\sqrt{n}} \Iprod{z}{g} 
			\label{eq:dy2}\\
	= & ~ \min_{g} \max_{\|y\|_1\leq 1,\|z\|_1\leq 1} \iprod{\Phi y + \frac{1}{\sqrt{n}} z }{g} - \iprod{y}{h} 
			\label{eq:dy3}\\
	= & ~  \max_{\|y\|_1\leq 1,\|z\|_1\leq 1} \min_{g } \iprod{\Phi y + \frac{1}{\sqrt{n}} z }{g} - \iprod{y}{h} 
			\label{eq:dy4}\\
	= & ~  \max_{\|y\|_1\leq 1, \|\Phi y\|_1\leq 1/\sqrt{n}} \iprod{y}{h} = \delta\left({1\over \sqrt{n}}\right),
			\label{eq:dy5}
	\end{align}
where in~\eqref{eq:dy2a} we write $\min$ since the optimization can clearly be restricted to a suitably large
$\ell_\infty$-box, \eqref{eq:dy2} is by the dual representation of norms,~\eqref{eq:dy3} is by linearity,~\eqref{eq:dy4}
is by von Neumann's minimax theorem for bilinear functions, and~\eqref{eq:dy5} follows since the inner minimization
in~\eqref{eq:dy4} is $-\infty$ unless $\Phi y + \frac{1}{\sqrt{n}} z = 0$. In fact, the equality of \prettyref{eq:dy2a} and \prettyref{eq:dy5} also follows from the strong duality of finite-dimensional LP.

To prove the left-hand bound in~\eqref{eq:dual}, we first introduce an auxiliary linear program:
	\begin{equation}
	\tilde{\delta}(t) \triangleq \max_\Delta \{\iprod{\Delta}{h}: \|\Delta\|_1\leq 1, \iprod{\Delta}{\ones}=0,
	\|\Phi \Delta\|_1\leq t\}\,,
	\label{eq:LP-Gt}
	\end{equation}
where $\Phi$ and $\Delta$ are as in~\eqref{eq:LP-G}. 

We recall the basics of the Le Cam's two-point method~\cite{Lecam86}. If there exist two priors on $\Theta$, under which the distributions of the samples are not perfectly distinguishable, i.e., of a small total variation, then the separation in the functional values constitutes a lower bound that holds for all estimators.
Specifically, for estimating linear functionals under the quadratic risk, we have \cite{Lecam86}
\begin{equation}
R_n \geq \frac{1}{8}  \max_{\pi,\pi'} \iprod{\pi-\pi'}{h}^2 (1-\TV((\Phi \pi)^{\otimes n},(\Phi \pi')^{\otimes n}),
\label{eq:lc-tv}
\end{equation}
where $\pi,\pi'$ are probability vectors on $\Theta$.
 Since  $\TV((\Phi \pi)^{\otimes n},(\Phi \pi')^{\otimes n}) \leq n \TV(\Phi \pi,\Phi \pi')$, we have
\begin{equation}\label{eq:dy7}
R_n \geq \frac{1}{16} \pth{ \max_{\pi,\pi'}\sth{\iprod{\pi-\pi'}{h}: \TV(\Phi \pi,\Phi \pi') \leq \frac{1}{2n}} }^2.
\end{equation}
Now, we observe that 
\begin{equation}\label{eq:dy6}
	\max_{\pi,\pi'}\sth{\iprod{\pi-\pi'}{h}: \TV(\Phi \pi,\Phi \pi') \leq \frac{1}{2n}} = 2 \tilde \delta \left({1\over
2n}\right)\,.
\end{equation}
Indeed, the optimizer $\Delta$ in~\eqref{eq:LP-Gt} can clearly be chosen so that $\|\Delta\|_1=1$. Next, decompose $\Delta =
\Delta_+ - \Delta_-$, where $\Delta_{\pm}(x) \eqdef \max(\pm\Delta(x), 0)$. From $\iprod{\Delta}{1}=0$ we conclude
that $\pi=2\Delta_+$ and $\pi'=2\Delta_-$ are valid probability distributions on $\Theta$ and, furthermore,
$$ \TV(\Phi \pi, \Phi \pi') = \|\Phi \Delta\|_1,\quad \Iprod{\pi - \pi'}{h} = 2\iprod{\Delta}{h}\,.$$
For the reverse direction, simply take $\Delta = {1\over2}(\pi-\pi')$.

Overall, from~\eqref{eq:dy7}-\eqref{eq:dy6} we get
$$ R_n \geq {1\over 4} \tilde\delta\left({1\over 2n}\right)^2\,.$$
To complete the proof the lower-bound in~\eqref{eq:dual} we invoke property 2 (to convert to $\delta(t)$) and property 1 (with $\lambda = 1/2$) from the following lemma (proved in \prettyref{app:aux}).
\end{proof}

\begin{lemma}[Properties of $\delta(t)$ and $\tilde\delta(t)$]
\label{lmm:aux-dual}
\begin{enumerate}
	\item For any $\lambda \in [0,1]$ we have 
		$$ \delta(\lambda t) \ge \lambda \delta(t), \qquad \tilde \delta(\lambda t) \ge \lambda \tilde
		\delta(t)\,.$$
	\item Assuming that there exist $\theta_{\pm}$ such that $h(\theta_+)\ge0\ge h(\theta_-)$ we have 
		\begin{equation}\label{eq:deltas}
			\tilde\delta(t) \le \delta(t) \le 2 \tilde\delta(t)\,.
		\end{equation}		
	\item For any non-constant $h$, there exists $C=C(h)>0$ such that\footnote{This property is not used for establishing \prettyref{thm:dual}. We prove it because it implies that $1/n$ (parametric rate) is always a lower bound in~\eqref{eq:dual}.} 
		$$ \tilde \delta(t) \geq C(h) t.$$
\end{enumerate}
\end{lemma}

\begin{remark}
\label{rmk:highprob}	
	Although \prettyref{thm:dual} is in terms of the mean-square error, it is easy to obtain high-probability bound using standard concentration inequalities.
	Consider the estimator \prettyref{eq:hatF} with $g$ being the solution to the LP \prettyref{eq:LP-G} with $t=\frac{1}{\sqrt{n}}$. Since
	the bias satisfies $|\Expect[\hat F-F]| \leq  \delta(\frac{1}{\sqrt{n}})$ and the standard deviation is at most $\frac{1}{\sqrt{n}}\|g\|_\infty \leq  \delta(\frac{1}{\sqrt{n}})$, Hoeffding inequality implies the Gaussian concentration
	$\Prob[|\hat F-F| \geq t \delta(1/\sqrt{n})] \leq \exp(-c t^2)$ for all sufficiently large $t$ and some absolute constant $c$.
\end{remark}

\section{Lossy population recovery}
	\label{sec:bec}

Capitalizing on the general framework introduced in \prettyref{sec:general}, we prove the sample complexity bounds for lossy \pop announced in \prettyref{thm:bec}. 
The outline is the following:
\begin{enumerate}
	\item In \prettyref{sec:bec-horo} we obtain sharp bounds on the value of the LP \prettyref{eq:LP-G} in the infinite-dimensional case, with the restriction that the estimator is allowed to depend only on the number of ones in the erased samples (\prettyref{prop:mintv-bec}). In particular, the upper bound on the LP value leads to the sample complexity upper bound in \prettyref{thm:bec} for any dimension $d$.
	\item The lower bound is proved in \prettyref{sec:lp-bec} in two steps:
	(i) By resorting to the Hellinger distance, in \prettyref{lmm:bhel} we remove the square-root gap in the general \prettyref{thm:dual}. 
	(ii) Recall from \prettyref{rmk:type} that it is sufficient to consider estimators that are functions of the number of zeros and ones. To complete the proof, we show that the number of zeros provides negligible information for estimating the probability of the zero vector.	
	\end{enumerate}

\subsection{Solving the linear programming by $H^\infty$-relaxation}
\label{sec:bec-horo}

In this subsection we determine the value of the LP \prettyref{eq:LP-G} for lossy \pop where the estimator is restricted to be functions on the output Hamming weight. 
In view of \prettyref{eq:bec}, to apply \prettyref{thm:dual}, 
here $\Phi$ is the transition matrix from the input to the output Hamming weight:
\begin{equation}
\Phi_{ij} = 
P(w(Y)=i |w(X)=j )=
\begin{cases}
\binom{j}{i} (1-\epsilon)^i \epsilon^{j-i}  & i \leq j\\
 0 & i > j\\
\end{cases}
\label{eq:Phi-bec}
\end{equation}
 and the linear functional corresponds to $h=e_0$. In other words, \prettyref{eq:LP-G} reduces to 
\begin{equation}
\delta(t) = \sup_\Delta \sth{\Delta_0: \|\Delta\|_1\leq 1, \left\|\sum_{j\geq 0} \Delta_j \mathrm{Bin}(j,1-\epsilon)\right\|_1\leq t}.
\label{eq:LP-G1}
\end{equation}

For notational convenience, let us define the following ``min-TV'' LP that is equivalent to~\eqref{eq:LP-G1}:
\begin{equation}
t(\delta) \eqdef \inf_\Delta\left\{\left\|\sum_{j\geq 0} \Delta_j \mathrm{Bin}(j,1-\epsilon)\right\|_1: \Delta_0 \ge \delta, \|\Delta\|_1 \le 1
\right\}\,.
\label{eq:mdelta}
\end{equation} 
Clearly we have $\delta(t(\delta))=\delta$. 

\begin{prop} 
\label{prop:mintv-bec}
If $\epsilon \leq \frac{1}{2}$, then
\begin{equation}
\delta \leq t(\delta) \leq 2(1-\epsilon) \delta.
\label{eq:mintv-bec1}
\end{equation}
If $\epsilon > \frac{1}{2}$, then
\begin{equation}
\delta^{\frac{\epsilon}{1-\epsilon}}  \leq t(\delta) \leq 
\pth{e^2 \delta_1 \log \frac{1}{\delta_1} }^{\frac{\epsilon}{1-\epsilon}}
\label{eq:mintv-bec2}
\end{equation}
where the right inequality holds provided that $\delta_1 \eqdef {\delta \over 1-\epsilon} < e^{-1}$.
\end{prop}

\begin{remark}
\label{rmk:Ggamma-bec}	
Thanks to \prettyref{thm:dual}, the lower bound on $t(\delta)$ in \prettyref{prop:mintv-bec} immediately translates into the following upper bound on the MSE of estimating $P_0$ in any dimension:
\[
\sup_{d \geq 1} R^*(n,d) \leq \delta(1/\sqrt{n})^2 \leq n^{- \min\sth{\frac{1-\epsilon}{\epsilon},1}},
\]
without any hidden constants. 
The upper bound on $t(\delta)$ require additional work to yield matching minimax lower bounds; this is done in \prettyref{sec:lp-bec}.
\end{remark}

\begin{proof}[Proof of \prettyref{prop:mintv-bec}]
Let $D$ be the open unit disk and $\bar D$ the closed unit disk on the complex plane. For analytic
functions on $D$ we introduce two norms:
	$$ \|f\|_{H^\infty(D)} \triangleq \sup_{z\in D} |f(z)|, \qquad \|f\|_{A} \triangleq \sum_{n\ge 0} |a_n|\,,$$
	where $a_n$ are the Taylor coefficients of $f(z) = \sum_n a_n z^n$. Functions with bounded $A$-norm form a space
	known as the Wiener algebra (under multiplication). It is clear that every $A$-function is also continuous on
	$\bar D$. Furthermore, by the maximal modulus principle,
	any function continuous on $\bar D$ and analytic on $D$ satisfies:
	\begin{equation}
	\|f\|_{H^\infty(D)} = \sup_{z \in \partial D} |f(z)|\,,
	\label{eq:maxmod}
	\end{equation}	
	and therefore:
		\begin{equation}\label{eq:bd1}
			\|f\|_{H^\infty} \le \|f\|_{A}\,.
\end{equation}		
	In general, there is no estimate in the opposite direction;\footnote{That is, the space $A$ is a strict subset of
	$H^\infty(D)$. This is easiest to see by noticing that all functions in $A$ are continuous on the unit
	circle, while $f(z) = {1\over (1-z)^i}$ is in $H^\infty$ but discontinuous on the unit circle. Another 
	example of a function in $H^\infty$ but not in $A$ is
	$f(z)=\exp({1+z\over 1-z})$, but it takes effort to show it is not in $A$. Also note that the linear
	functional $f \mapsto f(1)$ defined on polynomials is bounded in both $A$-norm and $H^\infty(D)$ norm.
	However, in $A$ it admits a unique extension to all of $A$, while in $H^\infty(D)$ there are different
	incompatible extensions (existence is from Hahn-Banach). }
	nevertheless, we can estimate the $A$-norm using the $H^\infty$-norm over a larger domain: Suppose $f\in H^\infty(rD)$ for $r>1$, then 
	\begin{equation}\label{eq:bd2}
		\|f\|_{A} \le \frac{r}{r-1} \|f\|_{H^\infty(rD)}\,.
	\end{equation}	
	Indeed, from Cauchy's integral formula we estimate coefficients $a_n$ as
		\begin{equation}
		 |a_n| \le r^{-n} \|f\|_{H^\infty(rD)}
		\label{eq:cauchy}
		\end{equation}
	and then sum over $n\ge0$ to get~\eqref{eq:bd2}.

	Now, to every sequence $\{\Delta_j: j\ge 0\}$ with finite $\ell_1$-norm, we associate a function in $A$ as
	$$ f(z) \eqdef \sum_{j\ge 0} \Delta_j z^j\,,$$
	which, in case when $\Delta$ is a probability distribution, corresponds its generating function.
	Note that the channel $\Phi$ maps the distribution of the input Hamming weight $w(X)$ into that of the output Hamming weight $w(Y)$ via a linear transformation.
	Equivalently, we can describe the action $\Phi$ on the function $f$ in terms how input generating functions are mapped to that of the output as follows:
\[
\Expect[z^{w(X)}] \overset{\Phi}{\mapsto}  \Expect[z^{w(Y)}]=\Expect[z^{\Binom(w(X),\bar\epsilon)}] = \Expect[(\epsilon z+ \bar\epsilon)^X]
\]
that is, $f(z) \mapsto f(\epsilon z+ \bar\epsilon)$. To this end, define the following operator, also denoted by $\Phi$, as
			\begin{equation}
			(\Phi f) (z) \eqdef f(\epsilon + \bar\epsilon z)\,,
			\label{eq:bec-op}
			\end{equation}
	which is known as a composition operator on a unit disk.\footnote{Despite the simple definition, characterizing properties of such
	operators is a rather difficult task cf.~\cite{cowen1988linear}.} Therefore, the channel maps $\Delta$ linearly into another $\ell_1$-sequence $\Delta' = \sum_j \Delta_j
	\mathrm{Bin}(j,\bar{\epsilon})$, which are exactly the coefficients of $\Phi f$.
		Indeed, $\Delta_j' = \sum_{k\geq j} \Delta_k \binom{k}{j} \bar\epsilon^j \epsilon^{k-j}$ and 
		\[
		\sum_{j \geq 0} \Delta_j' z^j = \sum_{k \geq 0} \Delta_k \sum_{j=0}^k \binom{k}{j} (\bar{\epsilon} z)^j \epsilon^{k-j} = \sum_{k \geq 0} \Delta_k (\bar{\epsilon} z+\epsilon)^k = (\Phi f)(z).
		\]
	Then, with the above identification, the LP \prettyref{eq:mdelta} can be recast as
	\begin{equation}\label{eq:mdef_2}
		t(\delta) = \inf \{\|\Phi f\|_{A}: \|f\|_{A} \le 1, f(0)\ge\delta\}\,.
\end{equation}	
	From \prettyref{eq:bec-op} an important observation is that $\Phi f$ restricts $f$ to a \textit{horodisk} 
	\begin{equation}\label{eq:br4}
		D_\tau  \eqdef \bar\tau+ \tau D = \{z \in \mathbb{C}: |z - \bar\tau| \leq \tau\}
\end{equation}	
which shrinks as $\tau$ decreases from 1 to 0 (see \prettyref{fig:horo}). Thus, $\|\Phi f\|_{H^\infty(D)} = \sup_{z\in D_{\bar{\epsilon}}} |f(z)|$.

\begin{figure}[ht]%
\centering
\begin{tikzpicture}[scale=2,font=\scriptsize,>=latex']
\draw[blue,thick] (0,0) circle (1);
\draw[red,thick] (1/2,0) circle (1/2);
\draw[thick] (1/3,0) circle (2/3);
\node at (0.8,0.95) {$D=D_1$};
\node at (-0.2,0.65) {$D_{2/3}$};
\node at (0.5,0.3) {$D_{1/2}$};
\draw[-latex] (-1.4,0) -- (1.4,0);
\draw[-latex] (0,-1.4) -- (0,1.4);
\end{tikzpicture}    
\caption{Horodisks.}%
\label{fig:horo}%
\end{figure}

In view of~\eqref{eq:bd1}, it is clear that $t(\delta)$ is lower bounded by
	\begin{equation}\label{eq:br3}
		t_1(\delta) \eqdef \inf \{\|\Phi f\|_{H^\infty(D)}: \|f\|_{H^\infty(D)} \le 1, f(0)\ge\delta\}\,,
	\end{equation}	
	Furthermore, if $\epsilon \leq 1/2$ then $0\in \bar D_{\bar{\epsilon}}$ and thus $t_1(\delta) \ge \delta$.
	This proves the lower bound in \prettyref{eq:mintv-bec1}. To show the upper bound, Take $f(z) = \delta(1-z)$, which is feasible since $f(0)=\delta$ and $\|f\|_{A} = 2 \delta < 1$. Clearly, $(\Phi f)(z) = \delta\bar{\epsilon}(1-z)$ which gives $t(\delta) \le 2 \bar{\epsilon} \delta$.

	In the remainder of the proof we focus on the non-trivial case of $\epsilon > 1/2$.
	Note that the M\"obius transform $z \mapsto {z-1\over z+1}$ maps the right half-plane onto the $D$ so that: a) the imaginary axis gets mapped to the unit circle $\partial D$; b) the line $1 + i \mreals$ gets mapped to the horocircle $\partial D_{1/ 2}$ that
	passes through $0$; c) the line ${\epsilon\over \bar{\epsilon}} + i \mreals$ gets mapped to the horocircle $\partial D_{\bar{\epsilon}}$. Then by
	Hadamard's three-lines theorem (see, e.g., \cite[Theorem 12.3]{simon2011convexity}) we have that for any function $f \in H^\infty(D)$:
		\begin{equation}
		\sup_{z\in D_{1/2}} |f(z)| \le \left( \sup_{z\in D} |f(z)| \right)^{1-2\bar{\epsilon}\over \epsilon} 
							\left( \sup_{z\in D_{\bar{\epsilon}}} |f(z)| \right)^{\bar{\epsilon} \over \epsilon}\,.
		\label{eq:threehoro}
		\end{equation}							
	Since any feasible solution $f$ to \eqref{eq:br3} has $f(0) \ge \delta$ and $\|f\|_{H^\infty(D)} \leq 1$,  we conclude that 
		$$ t(\delta) \ge t_1(\delta) \ge \delta^{\epsilon\over \bar{\epsilon}}\,,$$
proving the lower bound in \prettyref{eq:mintv-bec2}. 

To show the upper bound, we demonstrate an explicit feasible solution for \prettyref{eq:mdef_2}. 
Choose $\alpha < 1$ such that 
\begin{equation}
\beta \triangleq \frac{\delta}{1-\alpha} \leq 1.
\label{eq:beta1}
\end{equation}
The main idea is to choose the function so that the comparison inequality \prettyref{eq:threehoro} is tight. To this end, recall that Hadamard three-lines theorem holds with equality for exponential function. 
This motivates us to consider the following mother function 
		$$ g(z) \eqdef \beta^{1+z\over 1-z}\,.$$
	Note that for any $\eta\in\reals$, for the horodisk $D_\eta$ defined in~\eqref{eq:br4}, 
	$z \mapsto \frac{1+z}{1-z}$ maps the horocircle $\partial D_\eta$ back to the straight line ${\bar{\eta}\over \eta} + i \reals$. Since $0 \leq \beta \leq 1$, 
		\begin{equation}
		\| g \|_{H^\infty(D_\eta)} 
		=  \beta^{\inf_{z\in \partial D_\eta} \text{Re}(\frac{1+z}{1-z})} 
		= \beta^{\bar \eta \over \eta}\,.
		\label{eq:gDeta}
		\end{equation}
			In particular, since $D=D_1$, by setting $\eta=1$ we get $\|g\|_{H^\infty(D)} = 1$.
	Next, for $\alpha \in (0,1)$ define the scaled function
		\begin{equation}
		f_\alpha(z) \eqdef (1-\alpha) g(\alpha z)\,,
		\label{eq:scale}
		\end{equation}		
		which is a feasible solution to \prettyref{eq:mdef_2}. Indeed, $f_\alpha(0) = (1-\alpha) g(0) = (1-\alpha)\beta = \delta$, by definition. 
	Furthermore, invoking the estimate~\eqref{eq:bd2} with $r=1/\alpha$, we have
		\begin{equation}
		\| f_\alpha \|_{A} \leq \frac{1}{1-\alpha} \| f_\alpha \|_{H^\infty(D/\alpha)} = \| g \|_{H^\infty(D)} = 1 \,.
		\label{eq:feas1}
		\end{equation}		
	Setting $\eta = 1-\alpha {\epsilon}$ and $r = \frac{\eta}{\alpha\bar{\epsilon}} > 1$ since $\alpha < 1$, we have again from the reverse estimate~\eqref{eq:bd2}	
		\begin{equation}
		\| \Phi f_\alpha \|_{A} \leq \frac{1}{1-\frac{1}{r}} \| \Phi f_\alpha \|_{H^\infty(rD)}
		= \frac{1-\alpha}{1-\frac{1}{r}} \| g \|_{H^\infty(\alpha \epsilon+\alpha\bar{\epsilon} rD)}
		= (1-\alpha\epsilon) \| g \|_{H^\infty(D_\eta)}.
		\label{eq:feas2}
		\end{equation}
By		\prettyref{eq:gDeta} we have
\begin{align}
\| \Phi f_\alpha \|_{A}
\leq & ~ (1-\alpha\epsilon) \beta^{\bar \eta \over 	\eta} \le
\exp\pth{\frac{\alpha\epsilon}{1-\alpha\epsilon} \log \frac{\delta}{1-\alpha}} \eqdef \exp(\epsilon G/(1-\epsilon))
\label{eq:feas3}
\end{align}
Denote $\bar \alpha = 1-\alpha$ and choose $\bar \alpha = {\bar \epsilon \over -\log
\delta_1}$. Note that constraint~\eqref{eq:beta1} corresponds to $\delta_1  \log \delta_1 \ge -1$ which is automatically
satisfied for any $\delta_1>0$.
\apxonly{\par 
Generally, the optimal choice of $\bar\alpha$ to minimize~\eqref{eq:feas3} is
	$ \bar \alpha = e \delta x^*$ where $x^* \log x^* = 1/(e\delta_1)$. For $\delta_1 > e^{-3}$ this yields the
	bound $t(\delta) \lesssim \delta \exp\{-(2e^3 \delta)^{-1}\}$. In the interesting regime of $\bar \epsilon =
	{c\over \log n}$ this only results in the lower bound on risk $R^*(n,\bar \epsilon) \gtrsim {1\over \log n}$,
	whereas the optimal bound is $\Omega(1)$.

	\par}
Apply the simple inequality
$ {1 \over 1-\alpha \epsilon} 
\ge {1\over \bar \epsilon} \left(1 - {\bar\alpha \over \bar\epsilon }\right)$ and observe that for $\delta_1 \le e^{-1}$
the term in parentheses is non-negative. Thus, we get 
$$ {\alpha \epsilon \over 1-\alpha \epsilon} \ge {\epsilon \over \bar \epsilon}\left(1 - {\bar\alpha \over
\bar\epsilon}\right) (1-\bar \alpha) \ge {\epsilon \over \bar \epsilon} \left(1 - {\bar\alpha \over
\bar\epsilon}\right)^2 \ge  {\epsilon \over \bar \epsilon} \left(1 - 2 {\bar\alpha \over
\bar\epsilon}\right)\,. $$
From here we have
\begin{equation}\label{eq:feas4}
	G = \frac{\alpha \bar \epsilon}{1-\alpha\epsilon} \log \frac{\delta}{\bar\alpha} \le \log\pth{\delta_1 \log {1\over \delta_1}} \pth{1-{2\over \log {1\over \delta_1}}} \le \log\pth{\delta_1 \log {1\over
\delta_1}} + 2 
\end{equation}
since $\log (\delta_1 \log  {1\over \delta_1}) < 0$ and 
$\log \log {1\over \delta_1}\ge 0$. Plugging the estimate of $G$ 
into~\eqref{eq:feas3} we obtain~\eqref{eq:mintv-bec2}.
\end{proof}

\begin{remark}
\label{rmk:ms13}	We compare the proof techniques of the lower bound part of \prettyref{prop:mintv-bec} with that
of \cite{MS13}.  The sample complexity bound \prettyref{eq:ms13} is also obtained by bounding the value of the dual LP
from above via the $H^\infty$-norm relaxation. The suboptimality of the exponent
$\frac{1}{1-\epsilon}\log\frac{2}{1-\epsilon}$ in \prettyref{eq:ms13} seems to stem from the application of the Hadamard
three-circle theorem, which is applicable to three concentric circles centered at the origin. In comparison, the sharp
result in \prettyref{prop:mintv-bec} is obtained by comparing the value of any feasible $f$ on three horocircles (see
\prettyref{fig:horo}), which are images of three horizontal lines under the M\"obius transform and Hadamard three-lines
theorem is readily applicable yielding the optimal exponent $\frac{\epsilon}{1-\epsilon}$.  The upper bound part of
\prettyref{prop:mintv-bec} is new. 
\end{remark}

\apxonly{
The following is the truncated version of \prettyref{prop:mintv-bec}.
\begin{lemma}\label{lem:horo} 
Fix $\epsilon > {1\over 2}$. For $d \in \naturals$, define the finite-dimensional version of \prettyref{eq:mdelta}:
\begin{align}
 t(\delta,d) \eqdef
 & ~ \inf\left\{\left\|\sum_{j=0}^d \Delta_j \mathrm{Bin}(j,1-\epsilon)\right\|_1: \Delta\in\reals^{d+1}, \Delta_0 \ge \delta, \|\Delta\|_1 \le 1
\right\} \nonumber \\
= & ~ \inf \{\|\Phi f\|_{A}: \|f\|_{A} \le 1, f(0)\ge\delta, \deg f \leq d\}. \label{eq:mdelta-d}
\end{align}
Then for any $\delta < 1/e$, there exists $d_0\in\naturals$ such that 
\begin{equation}
d_0 \le \frac{1}{1-\epsilon}  \log^2 {1\over \delta}
\label{eq:d0}
\end{equation}
and
\begin{equation}
t(\delta/2,d_0) \leq \pth{e \delta \log \frac{1}{\delta} }^{\frac{\epsilon}{1-\epsilon}}
\label{eq:horo}
\end{equation}
\end{lemma}
\begin{proof}
Let $\beta = \delta/(1-\alpha)$ with $\alpha \in (0,1)$ to be specified.
	Let $g(z) = \beta^{1+z\over 1-z}$ as in the proof of \prettyref{prop:mintv-bec}, whose Taylor expansion is given by $g(z) = \sum_{n\ge 0} a_n z^n$ . 
		Let 	$\tilde g(z) = \sum_{0 \le n \le d_0} a_n z^n$. 
	The proof follows the same program as that of \prettyref{prop:mintv-bec}, with $\tilde g$ replacing $g$.
	 Similar to \prettyref{eq:scale}, define 
	\[
	\tilde f_\alpha(z) \eqdef \frac{1}{2} (1-\alpha) \tilde g(\alpha z).
	\]
	Then we have
	$\tilde f_\alpha(0)=\delta/2$, since $\tilde g(0) = g(0)=\beta$.
	Furthermore, by \prettyref{eq:gDeta}, $\| g\|_{H^\infty(D_\eta)}\leq \beta^{(1-\eta)/\eta}$  and  in particular $\|g\|_{H^\infty(D)} = 1$.
	By Cauchy's inequality \prettyref{eq:cauchy}, the Taylor series coefficients of $g$ satisfy $|a_n| \leq \|g\|_{H^\infty(D)} = 1$ and hence	
\begin{equation}\label{eq:dbl5}
			\|\tilde g(\alpha z) - g(\alpha z)\|_{A} \le {\alpha^{d_0+1}\over 1-\alpha} \eqdef \kappa\,.
\end{equation}		
	We later will choose $d_0$ so that $\kappa \le \beta^{(1-\eta)/\eta} \leq 1$. 
	Since $\|(1-\alpha) g(\alpha z)\|_A \leq \|g\|_{H^\infty(D)} = 1$ by \prettyref{eq:feas1}, we have 
	$\|\tilde f_\alpha\|_{A} \leq \|(1-\alpha) g(\alpha z)\|_{A}/2 + \|g(\alpha z) - \tilde g(\alpha z)\|_{A}/2 \leq 1$ 
	and hence $\tilde f_\alpha$ is a feasible solution for $t(\delta/2,d)$. 
	
	To bound the value of the objective function,
	note that $\|\Phi \tilde f_\alpha\|_{A} \leq \|\tilde g\|_{H^\infty(D_\eta)}/2$ by \prettyref{eq:feas2}, where $\eta = 1-\alpha \epsilon$.
	Since $D_\eta \subset D$ we have 
	$\|\tilde g\|_{H^\infty(D_\eta)} \leq \|\tilde g\|_{H^\infty(D_\eta)} + \|g-\tilde g\|_{H^\infty(D)} \leq \beta^{(1-\eta)/\eta} + \kappa \leq 2 \beta^{(1-\eta)/\eta}$
	and hence $\|\Phi \tilde f_\alpha\|_{A} \leq \beta^{(1-\eta)/\eta}$.
	Again, choose $\alpha = 1 - \frac{1}{\log(1/\delta)}$ gives the same upper bound as in \prettyref{eq:mintv-bec2}. 	
	The choice of $d_0$ in \prettyref{eq:d0} ensures
	\[
	\kappa=\frac{\alpha^{d_0+1}}{1-\alpha} \leq \beta^{(1-\eta)/\eta}=\pth{\frac{\delta}{1-\alpha}}^{\frac{\alpha \epsilon }{1-\alpha \epsilon }},
	\]
	in view of the fact that $\log\frac{1}{1-b}\geq b$ for $b\in(0,1)$ and $\log x < x$ for $x>0$. 
	 \nb{details: need 
	$d_0 + 1 \geq (\log \frac{1}{b} + \frac{\alpha \epsilon }{1-\alpha \epsilon } \log \frac{1}{\delta}) \frac{1}{\log \frac{1}{1-b}}$, where $b=1/\log(1/\delta)$.
	Then bound $\log\log$ by $\log$ and $\frac{\alpha \epsilon }{1-\alpha \epsilon }$ by $\frac{\epsilon }{1-\epsilon }$.}
\end{proof}
}

\subsection{Tight statistical lower bounds}
\label{sec:lp-bec}
Based on the general theory in \prettyref{thm:dual}, to prove the minimax lower bound announced in \prettyref{thm:bec}, there are two things to fix:
\begin{enumerate}[(a)]
\item why numbers of $0$'s provides negligible information for estimating the probability of the zero vector;
\item how to fix the square-root gap in the general lower bound in \prettyref{thm:dual}.
\end{enumerate}


\
We start with the second task. Recall the relation between the total variation and the Hellinger distance \cite{Tsybakov09}
\begin{equation}
\frac{1}{2}H^2\le \TV \le H\sqrt{1-H^2/4}.
\label{eq:TVH}
\end{equation}
and the tensorization property of the Hellinger distance 
\[
H^2(P^{\otimes n}, Q^{\otimes n}) = 2 -2 (1- H^2(P,Q)/2)^n.
\]
Applying both to the original Le Cam's method  \prettyref{eq:lc-tv}, after simple algebra we get the following minimax
lower bound (still for estimators based on the number of ones only):
\begin{equation}
\sqrt{R_n} \ge {1\over 4} \max_{\pi,\pi'}\left\{\pi(0)-\pi'(0): H^2(\Phi \pi,\Phi \pi') \leq {1\over 2n} \right\}\,.
\label{eq:hellinger}
\end{equation}
for any constant $C$, where $\Phi$ is given in \prettyref{eq:Phi-bec}, i.e.
$\Phi \pi = \Expect_{\theta\sim \pi}[ \Binom(\theta,\bar\epsilon)]$.

It remains to show that for specific models, e.g., lossy \pop, we have the following locally quadratic-like behavior: 
given the optimal $\Delta$ to \prettyref{eq:mdelta}, one can find feasible $\pi$ and $\pi'$ for \prettyref{eq:hellinger} so that $\pi-\pi'\approx \Delta$ and 
$H^2(\Phi \pi,\Phi \pi') \lesssim \TV(\Phi \pi,\Phi \pi')^2$, such that the lower bound in \prettyref{eq:TVH} is essentially tight.
This is done in the next lemma.
To construct the pair $\pi$ and $\pi'$, the main idea is to perturb a fixed distribution $\mu$ by $\pm \Delta$. In this case it turns out the center $\mu$ can be chosen to be a geometric distribution.
Furthermore, recall that the near-optimal solution used in \prettyref{prop:mintv-bec} deals with infinite sequence $\Delta$ in \prettyref{eq:mdelta},
 which, in the context of population recovery, corresponds to input strings of infinite length.\footnote{For the simple case of $\epsilon \leq 1/2$, the construction in \prettyref{prop:mintv-bec} is a degree-1 polynomial. This amounts to considering a single bit and using a pair of Bernoulli distributions to establish the optimality of the parametric rate $1/n$, which is standard.} 
It turns out that it suffices to consider $d = \Omega(\log^2\frac{1}{\delta})$.

\begin{lemma} 
\label{lmm:bhel}
Fix $\epsilon > {1\over 2}$ and $\delta < \frac{1-\epsilon}{e}$. Then 
there exists a pair of probability distributions $\pi$ and $\pi'$ on $\mathbb{Z}_+$ such that $|\pi(0) - \pi'(0)|\ge
\delta$ and
\begin{equation}\label{eq:bhel}
	H^2(\Phi \pi, \Phi\pi') \le C \pth{e^2\delta_1 \log \frac{1}{\delta_1} }^{\frac{2\epsilon}{1-\epsilon}},
\end{equation}
where $C =4$ and $\delta_1 \eqdef {\delta \over 1-\epsilon}$.
Furthermore, if $C=36$, both distributions can be picked to be supported on $\{0,\ldots,
d\}$ provided that $d \geq \frac{2\epsilon}{1-\epsilon}\log^2\frac{1}{\delta_1}$.
\end{lemma}
\begin{proof} Let $g,\beta,\alpha,b,\eta,r$ be as in the proof of \prettyref{prop:mintv-bec}. 
In particular, $\alpha=1-\frac{1-\epsilon}{\log \frac{1}{\delta_1}}$, $\eta=1-\alpha\epsilon$, $\beta=\delta_1\log\frac{1}{\delta_1}$, and $r=\frac{1-\alpha\epsilon}{\alpha\bar\epsilon}$.
Set
$$ f(z) = \bar\alpha g(\alpha z) - \bar\alpha g(\alpha)\,.$$
Let $\Delta_k \eqdef [z^k] f(z)$.
Then $\Delta_0=\bar\alpha g(0)-\bar\alpha g(\alpha) = \delta - \bar\alpha \beta^{{1+\alpha\over 1-\alpha}}$
and $\Delta_k = \bar\alpha \alpha^k [z^k] g(z)$ for $k\geq 1$.
We claim 
$$ \delta > \Delta_0 \ge {\delta\over 2}\,.$$
Indeed, the first inequality is clear, while the second is equivalent to $\beta^{2\alpha\over \bar \alpha} \le {1\over
2}$, which in turn follows from $\beta \le e^{-1}$ and $\alpha \ge 1/2$, both a consequence of the assumption 
$\delta_1 \le e^{-1}$.
Furthermore, recall from the proof of \prettyref{prop:mintv-bec} that $\|g\|_{H^\infty(D)} = 1$
and thus
\begin{equation}\label{eq:bh1}
	|\Delta_k| \le \bar \alpha \alpha^k\,,\qquad k \ge 1\,.
\end{equation}
Consider the following geometric distribution $\mu$ on $\mathbb{Z}_+$:
$$ \mu(k) \eqdef \bar \alpha \alpha^k\,. $$
Define now $\pi$ and $\pi'$ via
$$ \pi(k) \eqdef \mu(k) + \Delta_k\,, \quad \pi'(k) \eqdef \mu(k) - \Delta_k\,.$$
Note that $\mu(0) = \bar\alpha = {1-\epsilon\over \log {1\over \delta_1}} \ge \delta > \Delta_0$, which implies $\pi(0)\ge 0$ and
$\pi'(0)\ge 0$. Furthermore, from~\eqref{eq:bh1} we get that $\pi(k),\pi'(k)\ge 0$ for all $k\in\mathbb{Z}_+$. Since
$f(1)=\sum_k \Delta_k = 0$ we conclude that $\pi$ and $\pi'$ are indeed probability distributions satisfying
$$ \pi(0) - \pi'(0)=2\Delta_0 \ge \delta\,.$$
Next, notice that since $(\sqrt{1+r}-\sqrt{1-r})^2 \le 2r^2$ for all $r\in[0,1]$ we get
\begin{equation}\label{eq:bh4}
	H^2(\Phi\pi, \Phi\pi') = \sum_{k\ge 0} \left(
	\sqrt{\Phi\mu(k) + \Phi\Delta(k)} - \sqrt{\Phi\mu(k) - \Phi\Delta(k)}\right)^2  \le 2
		\sum_{k\ge 0} {\Phi\Delta(k)^2\over \Phi\mu(k)}\,.
\end{equation}		
Elementary calculation (e.g. from~\eqref{eq:bec-op} and $\sum_k z^k \mu(k) = {1-\alpha\over 1-\alpha z}$) shows that 
\begin{equation}\label{eq:bh2}
	\Phi\mu(k) = \left(1-{1\over r}\right) r^{-k}\,.
\end{equation}
We also know from the proof of \prettyref{prop:mintv-bec} and $g(\alpha) = \beta^{1+\alpha\over 1-\alpha} \le \beta^{\bar \eta\over \eta}$ that
\[\|\Phi f\|_{H^\infty(rD)} \le \bar \alpha \beta^{\bar \eta\over \eta} + \bar \alpha g(\alpha) \le 2 \bar
\alpha \beta^{\bar \eta\over \eta}\,.
\]
Therefore, from~\eqref{eq:cauchy} we get
\begin{equation}\label{eq:bh3}
	|\Phi\Delta(k)| = | [z^k] f(\bar\epsilon z + \epsilon) | \le2 \bar \alpha \beta^{\bar \eta\over \eta} r^{-k} \,.
\end{equation}
Plugging~\eqref{eq:bh2} and~\eqref{eq:bh3} into~\eqref{eq:bh4} we obtain as in~\eqref{eq:feas3}-\eqref{eq:feas4}:
\begin{equation}\label{eq:bh9}
	H^2(\Phi\pi, \Phi\pi') 
	\leq 4 \bar\alpha^2 \beta^{2\bar\eta/\eta} (1-1/r)^{-2} 
=	4 (1-\alpha \epsilon)^2 \beta^{2\bar\eta/\eta} 
	\le 4 \pth{e^2 \delta_1 \log \frac{1}{\delta_1} }^{\frac{2\epsilon}{1-\epsilon}}\,.
\end{equation}

To prove the second part, we replace $\pi,\pi'$ by their conditional version, denote by $\tilde \pi(k) = {\pi(k)\over \pi([0,d])}$ for $k\le d$ and similarly for $\tilde
\pi'$. 
Then 
\begin{equation}\label{eq:bh-tail}
\pi((d,\infty))+\pi'((d,\infty)) \le 2 \mu((d,\infty)) = 2 \sum_{k>d} \bar \alpha \alpha^k = 2\alpha^{d+1} \leq 2
\exp\pth{-\frac{d}{\log\frac{1}{\delta}}} \leq 2\delta_1^{2\epsilon\over 1-\epsilon},
\end{equation}
where the last inequality holds provided that $d \geq \frac{2\epsilon}{1-\epsilon}\log^2\frac{1}{\delta_1}$.
For Hellinger, 
\begin{align} H(\Phi\pi,\Phi\pi') &\ge H(\Phi\tilde \pi, \Phi \tilde \pi') -H(\Phi\pi, \Phi \tilde \pi) -
		H(\Phi\pi', \Phi \tilde \pi')\label{eq:bh5}\\
	&\ge H(\Phi\tilde \pi, \Phi \tilde \pi') -H(\pi, \tilde \pi) -
		H(\pi', \tilde \pi')\label{eq:bh6}\\
	&= H(\Phi\tilde \pi, \Phi \tilde \pi') - \sqrt{2-2\sqrt{\pi([0,d])}} - \sqrt{2-2\sqrt{\pi'([0,d])}}\
		\label{eq:bh7}\\
	&\ge H(\Phi\tilde \pi, \Phi \tilde \pi') - 4\delta_1^{\epsilon\over 1-\epsilon}\,,\label{eq:bh8}
\end{align}
where~\eqref{eq:bh5} is from the triangle inequality for Hellinger distance,~\eqref{eq:bh6} is from the data processing inequality for the
latter,~\eqref{eq:bh7} is explicit computation and finally in~\eqref{eq:bh8} 
we used \prettyref{eq:bh-tail} and
the fact that $\sqrt{2-2\sqrt{1-x}} \leq \sqrt{2x}$ for all $0\leq x \leq 1$.
In view of \prettyref{eq:bh9}, this completes the proof of \prettyref{eq:bhel} with $C=36$.
\end{proof}

\medskip

Finally, we put everything together.

\begin{proof}[Proof of Theorem~\ref{thm:bec}] The upper-bound on the sample complexity follows from
Theorem~\ref{thm:dual} and Proposition~\ref{prop:mintv-bec}. 

For the lower bound, we need to show that the number of zeros carries almost no information. 
This is done by dimension expansion. Indeed, set $d'\gg d$ and extend $\pi$ and $\pi'$ to distributions on
$\{0,\ldots,d'\}$ by zero-padding.
The intuition is that if the input vector contains at most $d$ ones, then the number of zeros of the output is
distributed approximately as $\Binom(d',\bar\epsilon)$, almost independent of the input. We make this idea precise next.

Fix $\delta=4\delta'>0$, set $\delta_1 = {\delta \over 1-\epsilon}$ and $d ={2\epsilon\over
1-\epsilon} \log^2 {1\over \delta_1}$ and consider probability distributions
$\pi$ and $\pi'$ constructed in Lemma~\ref{lmm:bhel} with $H^2(\Phi \pi, \Phi \pi')\le h_1$, where
\begin{equation}\label{eq:bx1}
	h_1 = C \pth{e^2 \delta_1 \log \frac{1}{\delta_1} }^{\frac{2\epsilon}{1-\epsilon}}\,.
\end{equation}
Take $d'=\frac{16 \bar \epsilon d^2}{\epsilon h_1}$ and note that
according to Lemma \ref{lmm:bin} in \prettyref{app:aux} we have
\begin{equation}\label{eq:bx2}
	H^2(\Binom(d'-s,\bar\epsilon), \Binom(d',\bar\epsilon)) \leq \frac{4 \bar\epsilon d^2}{\epsilon d'}  = {h_1\over 4}\qquad \forall s\in\{0,\ldots,d\}\,.
\end{equation}
It suffices to show
\begin{equation}
	H^2( \Expect_{\theta\sim \pi}[ \Binom(\theta,\bar\epsilon)\otimes \Binom(d'-\theta,\bar\epsilon) ],
	\Expect_{\theta\sim \pi'}[ \Binom(\theta,\bar\epsilon)\otimes \Binom(d'-\theta,\bar\epsilon) ])	\le 4h_1\,,
\label{eq:hellinger-mix}
\end{equation}
Indeed, assuming~\eqref{eq:hellinger-mix} we can conclude from~\eqref{eq:hellinger} that 
	$$ n^*_{\mathsf{L}}(\delta',d') \ge {1\over 8 h_1}\,,$$
	which in view of~\eqref{eq:bx1} is precisely the statement of the theorem.
To show~\eqref{eq:hellinger-mix} consider the following chain:
\begin{align}
	& ~ H(\Expect_{\theta\sim \pi} [\Binom(\theta,\bar\epsilon) \otimes \Binom(d'-\theta,\bar\epsilon)], \Expect_{\theta\sim \pi'} [\Binom(\theta,\bar\epsilon) \otimes \Binom(d'-\theta,\bar\epsilon)])
	\nonumber \\
\leq & ~ H(\Expect_{\theta\sim \pi} [\Binom(\theta,\bar\epsilon)], \Expect_{\theta\sim \pi'}
[\Binom(\theta,\bar\epsilon)]) + \sqrt{h_1} \label{eq:bx3}\\
= & H(\Phi \pi, \Phi \pi') + \sqrt{h_1} \le 2\sqrt{h_1}
\end{align}
where~\eqref{eq:bx3} is an application of the following Lemma with $\tau=\sqrt{h_1}/2$.
\end{proof}

\begin{lemma}
\label{lmm:Hprod}	
	Suppose there exists $Q^*$ such that $H(Q^*,Q_\theta) \leq  \tau$ for any $\theta\in\Theta$. Then for any distributions $\pi$ and $\pi'$ on $\Theta$,
	\[
	H(\Expect_{\theta\sim \pi} [P_\theta \otimes Q_\theta],\Expect_{\theta\sim \pi'} [P_\theta \otimes Q_\theta]) \leq 
	H(\Expect_{\theta\sim \pi} P_\theta, \Expect_{\theta\sim \pi'} P_\theta) +	2\tau.
	\]
\end{lemma}
\begin{proof}
 	The triangle inequality yields
	\begin{align*}
		H(\Expect_{\theta\sim \pi} [P_\theta \otimes Q_\theta],\Expect_{\theta\sim \pi'} [P_\theta \otimes Q_\theta]) 
\leq & ~ 	H(\Expect_{\theta\sim \pi} [P_\theta \otimes Q^*],\Expect_{\theta\sim \pi'} [P_\theta \otimes Q^*]) \\
& ~ +	H(\Expect_{\theta\sim \pi} [P_\theta \otimes Q_\theta],\Expect_{\theta\sim \pi} [P_\theta \otimes Q^*]) \\
& ~ +	H(\Expect_{\theta\sim \pi'} [P_\theta \otimes Q_\theta],\Expect_{\theta\sim \pi'} [P_\theta \otimes Q^*]).
	\end{align*}
	Here 
	\[
	H(\Expect_{\theta\sim \pi} [P_\theta \otimes Q^*],\Expect_{\theta\sim \pi'} [P_\theta \otimes Q^*]) = H(\Expect_{\theta\sim \pi} [P_\theta] \otimes Q^*,\Expect_{\theta\sim \pi'} [P_\theta] \otimes Q^*) = H(\Expect_{\theta\sim \pi} [P_\theta],\Expect_{\theta\sim \pi'} [P_\theta])
	\]
	and, by convexity,
	$H(\Expect_{\theta\sim \pi} [P_\theta \otimes Q_\theta],\Expect_{\theta\sim \pi} [P_\theta \otimes Q^*]) \leq \Expect_{\theta\sim \pi}H(Q_\theta,Q^*) \leq \tau$.
\end{proof}

\apxonly{Note: this lemma and this result can also be derived from the fact that
	$$ 1-H^2(P_{XY}, P_{X'Y'})/2 \ge (1-H^2(P_{X, P_X'})/2) \min_{x,x'} (1-H^2(P_{Y|X=x, P_{Y'|X'=x'}})/2) $$}

\apxonly{
\subsection{Bounds for general estimators via analysis alone}

Among the two problems identified in the beginning of Section~\ref{sec:lp-bec} only the second one requires statistical tools (Hellinger distance).
Indeed, in this section we briefly outline a purely analytic resolution of the first problem.

Similarly to~\eqref{eq:mdef_2} we now define
$$ t_d(\delta) = \inf_{\Delta}  \quad \sum_{x_0,x_1} \left|\sum_{\ell=0}^d \Delta_\ell {\ell \choose x_1} {d-\ell \choose x_0} \bar
\epsilon^{x_1+x_0} \epsilon^{d-x_0-x_1}\right|\,,$$
where the conditions on $\{\Delta_\ell: \ell = 0,\ldots d\}$ are as before: $\Delta_0=\delta$, $\|\Delta\|_1\le 1$. Defining
$$ f(z) \eqdef \sum_{\ell} \Delta_\ell z^\ell\,,$$
we can rephrase the problem as
$$ t_d(\delta) = \min \{\|\tilde\Phi_d f\|_A: \|f\|_A \le 1, f(1)=0, f(0)=\delta\}\,,$$
where the operator $\tilde \Phi_d$ maps a univariate polynomial $f$ of degree $d$ to a bi-variate polynomial by
the rule
\begin{equation}\label{eq:phi2}
	(\tilde\Phi_d f)(z_0,z_1) \eqdef (\bar \epsilon z_0 + \epsilon)^d f\left( \bar \epsilon z_1 + \epsilon \over \bar \epsilon
z_0 + \epsilon \right)\,,
\end{equation}
and $H^\infty$ and $A$-norms for bi-variate polynomials are defined as
\begin{align} \|\sum_{n_0,n_1} a_{n_0,n_1} z_0^{n_0} z_1^{n_1} \|_{A} &\eqdef \sum_{n_0,n_1} |a_{n_0,n_1}|\\
   \|g(z_0,z_1) \|_{H^\infty(D)} &\eqdef \sup_{z_0,z_1 \in D} |g(z_0,z_1)|
\end{align}

Bounding the value of $t_d(\delta)$ is done with the help of the following Lemma. Note that the $d_0$-degree polynomial
$f$ can be taken to be exactly the one constructed in the proof of Lemma~\ref{lem:horo}.

\nbr{YP needs to spell out the punchline of this result: (a) improve condition on $d$;  (b) cannot be combined with the Hellinger estimate so loses a factor of two in the exponent  }

\begin{lemma} Fix $\epsilon > 1/2$, then there exists a constant $C_\epsilon$ as follows. Let $f$ be a degree $d_0$
polynomial such that $\|f\|_{H^\infty(D)} \le 1$ and $\sup_{z\in D}|f(\bar\epsilon z +
\epsilon)| = \delta$. Then for $d = \lceil d_0 \log {1\over \delta} \rceil$ we have
	\begin{equation}\label{eq:beluga}
		\delta \le \|\tilde \Phi_d f\|_{A} \le C_\epsilon \cdot \delta  d_0^3  \log^2 {1\over
	\delta}\,.
\end{equation}	
\end{lemma}
\begin{proof} We will use the following relation between $A$ and $H^\infty$ norms:
\begin{equation}\label{eq:bld2}
	\|g\|_{H^\infty} \le \|g\|_A \le  d^2 \|g\|_{H^\infty}\,,
\end{equation}
for any degree-$d$ bivariate polynomial $g$.
For the lower bound in~\eqref{eq:beluga} we use first bound in~\eqref{eq:bld2} and simply lower-bound $H^\infty$ by taking $z_1\to1$. We proceed to the upper
bound.
Let us denote $f(z) = \sum_{n\le d_0} a_n z^n$ and 
$$ M_d \eqdef d^2 \|\tilde \Phi_d f\|_{H^\infty(D)}$$
By~\eqref{eq:bld2} we need to show an upper-bound on $M_d$ only.

We will use the following property: For any $r>1$ and any degree-$d_0$ polynomial $p$ we have
\begin{equation}\label{eq:blg1}
		\sup_{z\in D} |p(zr)| \le r^{d_0} d_0 \sup_{z\in D} |p(z)|\,,
\end{equation}	
which follows from an obvious $|[z^k] p(z)| \le \sup_{z\in D} |p(z)|$.

Fix some $\theta \in (0,1)$ (it will be taken close to 1).  We will compute $H^\infty(D)$-norm of $\tilde \Phi_d f$ by
fixing $z_0$ and maximizing over $z_1$. We separate two cases: 

Case 1. If $\theta < |\bar \epsilon z_0 + \epsilon| \le 1$ then we first notice that disk described by the image of 
	$$ z_1 \mapsto { \bar \epsilon z_1 + \epsilon \over \bar \epsilon z_0 + \epsilon} $$
	as $z_1$ ranges over $D$
	is contained inside the disk
	$$ z_1 \mapsto \bar \epsilon z_1 + \epsilon$$
	as $z_1$ ranges over ${1-\theta \epsilon\over 1-\epsilon} D$ (\textbf{TODO:} Mistake! This is only true when
	$z_0$ is real, so need to show that $z_0$ real is the worst case!). Thus by~\eqref{eq:blg1} we get
	$$ \sup_{z_1 \in D} d^2 \left|(\bar \epsilon z_0 + \epsilon)^d
	f\left( \bar \epsilon z_1 + \epsilon \over \bar \epsilon z_0 + \epsilon \right)\right| \le
		d^2 d_0 \left({1-\theta \epsilon\over 1-\epsilon}\right)^{d_0} \delta $$
	and in view of ${1-\theta \epsilon\over 1-\epsilon} \le \theta^{-{\epsilon\over \bar\epsilon}-1}$ we can
	continue as
	$$ \sup_{z_1 \in D} d^2 \left|(\bar \epsilon z_0 + \epsilon)^d
	f\left( \bar \epsilon z_1 + \epsilon \over \bar \epsilon z_0 + \epsilon \right)\right|\le d^2 d_0 \theta^{-c d_0} \delta $$
	for some $c>1$. 

Case 2. If $|\bar\epsilon z_0 + \epsilon| \le \theta$ then we upper bound 
	$$ \sup_{z_1 \in D} d^2 \left|(\bar \epsilon z_0 + \epsilon)^d
	f\left( \bar \epsilon z_1 + \epsilon \over \bar \epsilon z_0 + \epsilon \right)\right| \le
		d^2 \sup_{z_1 \in D} \sum_{n\le d_0} |a_n| \theta^{d-n} |\bar \epsilon z_1 + \epsilon|^n \le d^2 d_0
		\theta^{d-d_0} \|f\|_{H^\infty(D)} \le d^2 d_0 \theta^d \theta^{-cd_0}\,.$$
	where the last step is due to $c>1$.
	
	Thus, to summarize we get:
	$$ M_d \le 
		d^2 d_0 \theta^{-cd_0} \cdot \max \left\{\theta^d, \delta\right\}\,.$$
	Now choose $d$ as in statement of Lemma and $\theta$ so that
$\theta^d = \delta$.
	Then we get $\theta^{-c d_0} = e^c$ and thus
		$$ M_d \le d_0^3  \left(\ln {1\over \delta}\right)^2 e^c \delta\,, $$
	as claimed.
\end{proof}
}

\section{Noisy population recovery}
\label{sec:bsc}

In this section we bound the value of the LP \prettyref{eq:LP-G} for noisy \pop.
As previously mentioned in \prettyref{rmk:hellinger-preview}, 
since in this case the sample complexity turns out to grow super-polynomially, 
the general result in \prettyref{thm:dual} suffices to produce
the sample complexity bound announced in \prettyref{thm:bsc}
and there is no need to consider refined Hellinger-based minimax lower bound as developed in \prettyref{sec:lp-bec} for lossy \pop.

We begin by specializing the LP in \prettyref{eq:LP-G} to the noisy \pop setting with error probability $\epsilon$, where the transition matrix $\Phi$
is given by \prettyref{eq:bsc} throughout this section.
Recall from \prettyref{eq:bsc} that conditioned on the input Hamming weight $w(X)=j$, the 
output Hamming  weight $w(Y)$ is distributed as the convolution $\Binom(j,\bar\epsilon) * \Binom(d-j,\epsilon)$.
Thus, similarly to \prettyref{eq:mdelta}, we consider the equivalent formulation:
\begin{equation}
t(\delta,d) \eqdef \min_{\Delta\in\reals^{d+1}} \left\{\left\|\sum_{j=0}^d \Delta_j \Binom(j,\bar\epsilon) * \Binom(d-j,\epsilon) \right\|_1: \Delta_0 \ge \delta, \|\Delta\|_1 \le 1
\right\}\,.
\label{eq:mdelta-bsc}
\end{equation}
Since we can always negate the observed bits, in the sequel we shall assume, without loss of generality, that
\[
\epsilon < 1/2.
\]

To recast \prettyref{eq:mdelta-bsc} as an optimization problem in terms of functions (polynomials), we note that the channel $\Phi$ maps the generating function of the input weight to that of the output as follows:
\[
\Expect[z^{w(X)}] \overset{\Phi}{\mapsto}  \Expect[z^{w(Y)}]=\Expect[z^{\Binom(w(X),\bar\epsilon) + \Binom(d-w(X),\epsilon)}] = \Expect[(\bar\epsilon z + \epsilon)^{w(X)} (\epsilon z + \bar\epsilon)^{d-w(X)}  ]
\]
that is,
\begin{equation}
f(z) \mapsto (\Phi f)(z) \triangleq f\pth{\frac{\bar\epsilon z + \epsilon}{\epsilon z + \bar\epsilon}} (\epsilon z + \bar\epsilon)^d.
\label{eq:bsc-op}
\end{equation}
Thus, each degree-$d$ polynomial is mapped to another via
\[
\sum_{i=0}^d a_i z^i \mapsto \sum_{i=0}^d a_i (\bar\epsilon z + \epsilon)^i(\epsilon z + \bar\epsilon)^{d-i}.
\]
Therefore, using the $A$-norm introduced in \prettyref{sec:bec-horo}, we have:
\begin{equation}
t(\delta,  d) = \min \sth{\left\|f\pth{\frac{\bar\epsilon z + \epsilon}{\epsilon z + \bar\epsilon}} (\epsilon z + \bar\epsilon)^d \right\|_{A}: f(0) \geq \delta,
 \|f\|_{A} \le 1, \deg f \le d}\,.
\label{eq:mdef_2-bsc}
\end{equation}

As usual, by choosing $f\equiv\delta$, we have the trivial bound
\begin{equation}
t(\delta,  d) \leq \delta.
\label{eq:m-trivial}
\end{equation}
The next proposition provides a sharp characterization:

\begin{prop} 
\label{prop:mintv-bsc}

Assume that $\epsilon < \frac{1}{2}$ and $d\geq 1$. Define
\begin{equation}
\mu(\epsilon) \triangleq \frac{\epsilon(1-\epsilon)}{(1-2\epsilon)^2}.
\label{eq:mu}
\end{equation}
There exist absolute constant $c,c'$ such that the following hold:
For all $\delta < 1$, 
\begin{equation}
t(\delta,d) \geq \exp\sth{-c   \pth{d \mu(\epsilon) \log^2 \frac{e}{\delta}}^{1/3} }.
\label{eq:m-lb}
\end{equation}
For all $\delta < 1/3$, 
\begin{equation}
t(\delta,d) \leq 
\max\sth{
\exp(-(1-2\epsilon)^2d), 
\exp\pth{- c' \pth{d\mu(\epsilon) \log^2 \frac{1}{\delta}}^{1/3} }}.
\label{eq:m-ub}
\end{equation}
\end{prop}
Again, thanks to \prettyref{thm:dual}, \prettyref{eq:m-lb} and \prettyref{eq:m-ub} translate into the sample complexity lower and upper bound in \prettyref{thm:bsc}, respectively.

The rest of this section is devoted to proving \prettyref{prop:mintv-bsc} using $H^\infty$-relaxations (\prettyref{sec:bsc-horo}), with upper and lower bound shown in \prettyref{sec:bsc-lb} and \prettyref{sec:bsc-ub}, respectively.

\subsection{Two $H^\infty$-relaxations}
	\label{sec:bsc-horo}
	Compared to the analysis of lossy \pop, it turns out that the upper estimate \prettyref{eq:bd2} on $A$-norm using $H^\infty$-norm over a larger disk, which we relied on in \prettyref{sec:bec-horo} to deal with the composition operator \prettyref{eq:bec-op}, 
 is not useful for the more complicated operator \prettyref{eq:bsc-op} and more intricate analysis is thus required.
Specifically, we need to consider two $H^\infty$-relaxations of \prettyref{eq:mdef_2-bsc}, the latter of which retains the $A$-norm constraint:
\begin{align}
t_1(\delta, d) = & ~ \min \sth{\left\|f\pth{\frac{\bar\epsilon z + \epsilon}{\epsilon z + \bar\epsilon}} (\epsilon z + \bar\epsilon)^d \right\|_{H^\infty(D)}: f(0) \geq \delta,
 \|f\|_{H^\infty(D)} \le 1, \deg f \le d}\,  \label{eq:m1}\\
t_2(\delta, d)= & ~ \min \sth{\left\|f\pth{\frac{\bar\epsilon z + \epsilon}{\epsilon z + \bar\epsilon}} (\epsilon z + \bar\epsilon)^d \right\|_{H^\infty(D)}: f(0) \geq \delta,
 \|f\|_{A} \le 1, \deg f \le d}\,.	 \label{eq:m2}
\end{align}
Since 
$ \|p\|_{H^\infty(D)} \leq  \|p\|_{A}  \leq d \|p\|_{H^\infty(D)}$ for any degree-$d$ polynomials $p$, 
we have, in turn,
\begin{equation}
t_1(\delta, d) \le t(\delta,  d) \le  d t_2(\delta, d)\,.
\label{eq:mm1}
\end{equation}

Next we simplify \prettyref{eq:m1} and \prettyref{eq:m2} via a change of variable.
note that as long as $\epsilon \neq \frac{1}{2}$, 
the M\"obius transform $z \mapsto \frac{\bar\epsilon z + \epsilon}{\epsilon z + \bar\epsilon}$ maps the unit circle to itself.
Therefore
\begin{align*}
\left\|f\pth{\frac{\bar\epsilon z + \epsilon}{\epsilon z + \bar\epsilon}} (\epsilon z + \bar\epsilon)^d \right\|_{H^\infty(D)} 
\overset{\prettyref{eq:maxmod}}{=}  & ~  \sup_{z\in \partial D} 
\left|f\pth{\frac{\bar\epsilon z + \epsilon}{\epsilon z + \bar\epsilon}} (\epsilon z + \bar\epsilon)^d \right|
\\
= & ~ \sup_{w\in \partial D} 
\left|f(w)  h(w) \right| \overset{\prettyref{eq:maxmod}}{=} \|f h\|_{H^\infty(D)},
\end{align*}
where 
\[
h(w) \triangleq \left(1-2\epsilon\over \bar\epsilon - \epsilon w\right)^d
\]
is analytic on $\bar D$ since $\epsilon < 1/2$ by assumption. Hence \prettyref{eq:m1} and \prettyref{eq:m2} are equivalent to
\begin{align}
t_1(\delta, d) = & ~  \min \sth{\left\|f h\right\|_{H^\infty(D)}: f(0) = \delta,
 \|f\|_{H^\infty(D)} \le 1, \deg f \le d}, \label{eq:m11}\\
t_2(\delta, d) = & ~  \min \sth{\left\|f h\right\|_{H^\infty(D)}: f(0) = \delta,
 \|f\|_{A} \le 1, \deg f \le d}. \label{eq:m22}
\end{align}


\subsection{Lower bound on $t_1$} 
\label{sec:bsc-lb}

Rewriting $h(z)$ as $h(z) = \frac{1}{(1- c (z-1))^d}$, where $c\triangleq {\epsilon\over 1-2\epsilon}$, we have for all $\theta \in [-\pi,\pi]$,
\[
|h(e^{i\theta})| = (1 + 2c(1+c) (1-\cos \theta))^{-d/2} \geq (1 + c(1+c) \theta^2)^{-d/2} \geq e^{- \mu d \theta^2/2}
\]
where $\mu =\mu(\epsilon)$ is defined in \prettyref{eq:mu}.
Next, from~\cite[Corollary 3.2]{borwein1997littlewood} we conclude that for any feasible $f$ for \prettyref{eq:m11}, we have
$$ \sup_{\theta \in (-a/2,a/2)} |f(e^{i\theta})| \ge (\delta/e)^{c_1/a}\,, $$
where $c_1$ is an absolute constant. Thus, 
$$ \|fh\|_{H^\infty(D)} \geq  \sup_{\theta \in (-a/2,a/2)}  |f(e^{i\theta})| |h(e^{i\theta})|  \ge e^{- \mu d a^2/2 -
{c_1\over a} \log \frac{e}{\delta}}\,.$$
Setting $a = \left(2 c_1 \log \frac{e}{\delta} \over \mu d\right)^{1\over 3}$ we get the following non-asymptotic
bound:
$$ t_1(\delta,d) \ge \exp\sth{-c_2   \pth{\mu(\epsilon) d \log^2 \frac{e}{\delta}}^{1/3} }\,, $$
where $c_2$ is an absolute constant.

\subsection{Upper bound on $t_2$} 
\label{sec:bsc-ub}
In view of \prettyref{eq:mm1}, we will show that for all $\epsilon<1/2$, $\delta \in(0,1/3)$  and all $d \in \naturals$ we have 
\begin{equation}
t(\delta,d) \leq 
\max\sth{
\exp(-(1-2\epsilon)^2d), 
\pth{\frac{d^4 \mu}{\log^4 \frac{1}{\delta}} }^{1/3}
\exp\sth{- c' \pth{d\mu \log^2 \frac{1}{\delta}}^{1/3} }}.
\label{eq:mintv-bsc1}
\end{equation}
for some absolute constant $c'$.

First consider $d \leq \frac{1}{(1-2\epsilon)^2} \log \frac{1}{\delta}$. By the trivial bound \prettyref{eq:m-trivial}, we have
\begin{equation}
t(\delta,d) \leq \exp(-(1-2\epsilon)^2d).
\label{eq:case1}
\end{equation}
In the sequel we shall assume that 
\begin{equation}
d \geq \frac{1}{(1-2\epsilon)^2} \log \frac{1}{\delta}.
\label{eq:r}
\end{equation}
 To construct a near optimal solution for \prettyref{eq:m22}, we modify the feasible solution previously used in the
 proof of \prettyref{prop:mintv-bec}.
Let 
\begin{equation}\label{eq:g_bsc}
g(z) = (1-z)^2 \delta^{1+z\over 1-z}, 
\end{equation}
and $f(z) = g(\alpha z)$ with $\alpha \in (0,1)$ to be chosen.
Let $\tilde g$ be the degree-$d$ truncation of the Taylor expansion of $g(z) = \sum_n a_n z^n$, and $\tilde f(z) = \tilde g(\alpha z)$.
Note that $\tilde f(0)=\tilde g(0)=g(0)=\delta$. 
Next, instead of invoking \prettyref{eq:bd2} which estimates the $A$-norm by the $H^\infty$-norm over a bigger disk and turns out to be too loose here, 
the next lemma (proved in \prettyref{app:aux}) uses the $H^\infty$-norm of the derivative:
\begin{lemma}
\label{lmm:gnorm}	
For any $\delta > 0$, 
	$\|g\|_{A} \leq A(\delta) \triangleq 4(\log \frac{1}{\delta}+3)$.
\end{lemma}
Therefore 
\[
\|\tilde f\|_{A} =\sum_{n=0}^d \alpha^n |a_n| \leq \sum_{n \geq 0} |a_n| = \|g\|_{A} \leq A(\delta).
\]
By Cauchy's inequality \prettyref{eq:cauchy}, the Taylor series coefficients of $f$ satisfy $|a_n| \leq \|g\|_{H^\infty(D)} = 1$ and hence	
\begin{equation}\label{eq:dbl5x}
			\|\tilde f - f\|_{A} \le {\alpha^{d_0+1}\over 1-\alpha}\,.
\end{equation}

Next we bound the objective function.
Rewriting $h(z)$ as $h(z) = \frac{1}{(1- c (z-1))^d}$, where $c\triangleq {\epsilon\over 1-2\epsilon}$, we recall that 
$$ \|h\|_{H^\infty(D)} = 1\,.$$
Thus, we have 
\begin{align}
\|\tilde f h\|_{H^\infty(D)} 
\leq & ~ \|f h\|_{H^\infty(D)} + \|(f-\tilde f) h\|_{H^\infty(D)} \nonumber \\
\overset{\eqref{eq:dbl5x}}{\leq} & ~ \|g(\alpha z) h(z)\|_{H^\infty(D)} +  {\alpha^{d+1}\over
1-\alpha} \nonumber \\
 \leq & ~ 4 \|\delta^{1+\alpha z\over 1-\alpha z} h(z)\|_{H^\infty(D)} +{\alpha^{d+1}\over
1-\alpha}, \label{eq:fh}
\end{align}
where in the last step we used $\|(1-\alpha z)^2\|_{H^\infty(D)} \leq 4$.
The first term is bounded by the following lemma (proved in \prettyref{app:aux}):

\begin{lemma}
\label{lmm:third} 
If 
\begin{equation}
d \ge 2 \log \frac{1}{\delta} \max\sth{\frac{1}{\mu(1-\alpha)^3},  \frac{1}{1-\alpha}}
\label{eq:d1}
\end{equation}
where $\mu=\mu(\epsilon)$ is defined in \prettyref{eq:mu},
 then
\begin{equation}
\|\delta^{1+\alpha z\over 1-\alpha z} h(z)\|_{H^\infty(D)} = \delta^{1+\alpha \over 1-\alpha}\,.
\label{eq:third}
\end{equation}
 \end{lemma}

Finally, set
\begin{equation}
\alpha = 1-\pth{ \frac{2}{d \mu} \log \frac{1}{\delta}  }^{1/3},
\label{eq:alpha-bsc}
\end{equation}
which, in view of the assumption \prettyref{eq:r}, fulfills the condition \prettyref{eq:d1} in the \prettyref{lmm:third}. Combining \prettyref{eq:fh} and \prettyref{eq:third} yields
$ \|\tilde f h\|_{H^\infty(D)} \le \delta^{1+\alpha\over 1-\alpha} + {\alpha^{d+1}\over
1-\alpha}.
$
Since $\alpha^d \leq \exp(-(1-\alpha) d)$, the assumption \prettyref{eq:r} together with \prettyref{eq:alpha-bsc} implies that $d \geq \log \frac{1}{\delta}$ and further $\alpha^d \leq \delta^{1 \over 1-\alpha}$. Hence
 \begin{equation}
\|\tilde f h\|_{H^\infty(D)} \leq \frac{1}{1-\alpha} \delta^{1 \over 1-\alpha} 
= \pth{\frac{d\mu}{2 \log \frac{1}{\delta}} }^{1/3}
\exp\sth{- \pth{\frac{d\mu}{2} \log^2 \frac{1}{\delta}}^{1/3} }.
 \label{eq:obj}
 \end{equation}

In summary, for any $\delta \in (0,1)$, we have constructed $\tilde f$ such that $\tilde f(0)=\delta$, $\|\tilde f\|_{A} \leq A(\delta) = 4(\log \frac{1}{\delta}+3)$ and 
$\|\tilde f h\|_{H^\infty(D)}$ is bounded by \prettyref{eq:obj}.
Rescaling by $A(\delta)$, we conclude there exists a universal constant $c_2$, such that for any $\delta'=\frac{\delta}{A(\delta)} \in (0,\frac{1}{3})$, 
\[
t_2(\delta',d) \leq c_2 \pth{\frac{d\mu}{\log^4 \frac{1}{\delta'}} }^{1/3}
\exp\sth{- c_2 \pth{d\mu \log^2 \frac{1}{\delta'}}^{1/3} },
\]
which, in view of \prettyref{eq:mm1}, yields \prettyref{eq:mintv-bsc1}.

\section{Smoothed estimators for lossy population recovery}
	\label{sec:smooth}

\newcommand{\gun}{g^{\textrm{u}}} \newcommand{\gsm}{g^{\textrm{s}}}

In this section, we construct an explicit estimator for lossy \pop that
is optimal up to a factor of $2$ in the exponent. We start from the 
unbiased estimator for $P_0$ and then modify it via the \emph{smoothing technique} proposed in
\cite{OSW16}.

Recall that for the linear estimator $g$ in \prettyref{eq:hatp}, its bias is bounded by $ \|\Phi^\top g-e_0\|_\infty$ and 
the standard deviation is at most $\frac{1}{\sqrt{n}} \|g\|_\infty$, 
where the matrix $\Phi$ is given by \prettyref{eq:Phi-bec} which is an upper triangular matrix with non-zero diagonals.
As mentioned in \prettyref{sec:intro}, the unique unbiased estimator is a linear estimator with coefficients $\gun = (\Phi^\top)^{-1} e_0$. Direct calculation shows that 
$\gun_j = (-\frac{\epsilon}{\bar\epsilon})^j$ for $j =0,1,\ldots,d$.\footnote{This can also be obtained from consider the inverse operator of \prettyref{eq:bec-op}, which is again a composition operator $(\Phi^{-1} f) (z) = f((z-\epsilon)/ \bar\epsilon )$.}
%

Note that $\|\gun\|_\infty = \max(1,
(\epsilon/\bar\epsilon)^{d+1})$. Hence, for $\epsilon \leq 1/2$, $\|\gun\|_\infty = 1$ and the unbiased estimator has a sample complexity at most
$\cO(\delta^{-2})$. For $\epsilon > 1/2$, the coefficients increase
exponentially in $d$ which results in high variance. To
alleviate this issue, we modify the estimator via the smoothing
technique proposed in \cite{OSW16}. 
The main idea of smoothing is to introduce an independent random integer $L$, 
truncate the unbiased estimator after the $L\Th$ term, and average the truncated estimators according to an appropriately chosen distribution that 
aims to balance the bias and variance.
Equivalently, this amounts to multiplying the coefficients of the unbiased estimator with a tail probability that modulates the exponential growth.
To this end, define the smoothed estimator $\hat P_0^{\rm u}$
as a linear estimator \prettyref{eq:hatp} with the coefficient $g$ given by
\[
\gsm_i \triangleq  \gun_i \cdot \Pr(L \geq i).
\]
The following theorem gives the sample complexity guarantee for Poisson smoothing.
\begin{theorem}
\label{thm:linear}
Let $\epsilon > 1/2$.
Let $L$ be Poisson distributed with mean $\lambda = \frac{1-\epsilon}{3\epsilon-1} \log n$. Then
\[
\sup_{P\in\calP_d} \Expect_P[(\hat P_0^{\rm u}-P_0)^2] \leq 4 n^{-\frac{1-\epsilon}{3\epsilon-1}}.
\]
Therefore, the sample complexity of the smoothed estimator is at most $O\Big( \delta^{-\frac{2(3\epsilon-1)}{1-\epsilon}}\Big)$.
\end{theorem}
\begin{proof}
We first bound the variance by the moment generating function of $L$. Observe that
\begin{align}
\|\gsm\|_\infty 
= \max_i |\gun_i | \cdot \Pr(L \geq i)  = \max_i (\epsilon/\bar\epsilon)^i \cdot \Pr(L \geq i ) 
 \leq  \EE_L [(\epsilon/\bar\epsilon)^L]  =e ^{\lambda\frac{2\epsilon-1}{1-\epsilon}},
\label{eq:sm3}
\end{align}
where the inequality follows from the assumption that $\epsilon/\bar{\epsilon} > 1$.
To bound the bias term, note that 
$(\Phi^\top \gsm)_0 = 1$ and hence
\begin{equation}
\| \Phi^\top \gsm - e_0 \|_\infty \leq \max_{1 \leq j \leq d} |(\Phi^\top \gsm)_j|.
\label{eq:sm1}
\end{equation}
For any $j > 0$,
\begin{align}
(\Phi^\top \gsm)_j
 = \sum^d_{i=0} \Phi_{ij} \gsm_i 
 = \sum^j_{i=0} {j \choose i} \epsilon^j (-1)^i \Pr(L \geq i) \triangleq \epsilon^j f(j).
\label{eq:sm2}
\end{align}
Observe that
\begin{align}
f(j) 
& = \sum^j_{i=0} {j \choose i}(-1)^i \Pr(L \geq i) \nonumber \\
& = \sum^j_{i=0} \left({j-1 \choose i-1} + {j-1 \choose i}\right)(-1)^i \Pr(L \geq i) \nonumber \\
& = \sum^j_{i=0} {j-1 \choose i-1}(-1)^i \Pr(L \geq i) + \sum^j_{i=0} {j-1 \choose i}(-1)^i \Pr(L \geq i) \nonumber \\
& = \sum^{j-1}_{i=0} {j-1 \choose i}(-1)^{i+1} \Pr(L \geq i+1) + \sum^{j-1}_{i=0} {j-1 \choose i}(-1)^i \Pr(L \geq i) \nonumber \\
& = \sum^{j-1}_{i=0} {j-1 \choose i}(-1)^{i} \Pr(L =i). \nonumber
\end{align}
If $L \sim \text{Poi}(\lambda)$, then
\[
f(j) = \sum^{j-1}_{i=0} {j-1 \choose i}(-1)^{i} \Pr(L =i) = \sum^{j-1}_{i=0} {j-1 \choose i}
(-1)^i e^{-\lambda} \frac{\lambda^{i}}{i!} = e^{-\lambda}
L_{j-1}(\lambda),
\]
where $L_{k}(\lambda) = 	\sum_{i=0}^k \frac{(-\lambda)^i}{i!}  \binom{k}{i}$
 is the Laguerre polynomial of degree $k$.
Since $|L_k(\lambda)| \leq e^{\lambda/2}$ for all $k\geq 0$ cf.~\cite[22.14.12]{AS64}, we have $|f(j)|\leq e^{-\lambda/2}$ for all $j\geq 1$.
Hence by \prettyref{eq:sm1} and \prettyref{eq:sm2},
\[
\|\Phi^\top \gsm - e_0 \|_\infty \leq e^{-\lambda/2}.
\]
Combining the above equation with
\prettyref{eq:sm3} yields
\[
\|\Phi^\top \gsm - e_0 \|_\infty + \frac{\|\gsm\|_\infty}{\sqrt{n}}
\leq  e^{-\lambda/2} + \frac{e^{\lambda\frac{2\epsilon-1}{1-\epsilon}}}{\sqrt{n}}.
\]
The theorem follows by setting $\lambda = \frac{1-\epsilon}{3\epsilon-1} \log n$.
\end{proof}

\appendix

\section{Converting individual recovery to population recovery}
\label{app:alg}


In this appendix we describe the algorithm in \cite{DRWY12} that relates individual recovery to population recovery and converts any estimator for $P_0$ to a distribution estimator. The main ingredient is induction on the length $d$
and the observation that for any $x$, one can convert any estimator $\hat P_0$ for
$P_0$ to one for $P_x$:
\[
\wt P_x = \hat P_0(Y_1\oplus x, \ldots, Y_n\oplus x),
\]
which inherits the performance guarantee of $\hat P_0$ under both lossy and noisy model.

Let $xa$ denote the sequence obtained by concatenating sequences $x$
and $a$. Consider the following algorithm:
\begin{itemize}
\item Input $\delta$.
\item Initialize: $\cS_1 = \{0,1\}$.
\item For each $i$ from $2$ to $d$: 
\begin{itemize}
\item $\cS_i \leftarrow \cup_{x \in \cS_{i-1}} \{x0, x1\}$.
\item For each $x$ in $\cS_i$, compute the estimate $\wt P_x = \hat P_0(Y_1\oplus x, \ldots, Y_n\oplus x)$.  
	\item Let  $\cS_i \leftarrow \cS_i \setminus \{x: \wt P_x \leq 2\delta\}$.
\end{itemize}
\item 
Assign $\hat P_x = \wt P_x$ for all $x \in \calS_d$ and $\hat P_x = 0$ otherwise.
\end{itemize}

The following lemma, applicable to both lossy and noisy recovery, improves the original guarantee in \cite{DRWY12} which depends on the support size of the distribution.

 \begin{lemma}
   \label{lem:irptorp}
If there is an algorithm $\hat P_0$ that estimates $P_0$ to an
accuracy of $\delta$ using $n$ samples with probability at least $1-
\alpha$, then there is an algorithm that with probability $\geq 1 -
9d\alpha/(4\delta)$ satisfies,
\[
\max_{x \in \{0,1\}^d} |P_x - \hat{P}_x| \leq 4\delta.
\] 
The run time for the algorithm is $\cO(d/\delta \cdot (t+nd))$, where $t$
is the time it takes to compute $\hat P_0$.
\end{lemma}

\begin{remark}
\label{rmk:median}	
\prettyref{lem:irptorp} shows that the sample complexity of population recovery (estimating all probabilities) is within a logarithmic factor of that of individual recovery (estimating $\hat P_0$), namely, $n^*(d,\delta)$. To see this, consider an estimator $\hat P_0$ that achieves $\Expect(\hat P_0-P_0)^2 \leq \delta^2$ so that
$|\hat P_0-P_0| \leq 2 \delta$ with probability at least $1/4$. By the usual split-sample-then-median method,\footnote{That is, 
divide all the samples into $\log \frac{1}{\tau}$ batches, apply the same estimator to each batch and take the batchwise median.} $O(n^*(\delta,d) \cdot \log \frac{d}{\delta \tau})$ samples suffices to boost the probability of estimating $P_0$ within $\delta$ to $1- \frac{\delta\tau}{d}$, which, in view of \prettyref{lem:irptorp}, suffices to estimate all probabilities within $\delta$ with probability $1-\tau$. 		
\end{remark}


\begin{proof}
We first show that with high probability all sequences with
probability $\geq 4 \delta$ remain in $\cS_d$.  Observe that if a
sequence $x$ has probability $\geq 4\delta$, then all of its 
prefixes also has probability $\geq 4\delta$. By definition, $\Pr(|P_x - \hat{P}_x|\geq 2\delta) \leq \alpha$.
Hence, the probability that $x$ is discarded is at most
$\alpha (d-1)$ by the union bound and the probability that any such
sequence is absent in $\calS_d$ is at most $(d-1)\alpha/(4\delta)$.

We now show by induction that $|\cS_d| \leq 1/\delta$ with probability $\geq 1 -
2(d-1)\alpha/\delta$.  Suppose $|\cS_{i-1}| \leq
1/\delta$ with probability $\geq 1- 2(i-2)\alpha/\delta$, which holds
for $i-1 = 1$. Then,
\begin{align*}
\Pr(\exists y \in \cS_i : {P}_y \leq \delta) 
&\leq \sum_{y \in \cup_{x \in \cS_{i-1}} \{x0, x1\}} \Pr(\hat{P}_y > 2\delta, P_y \leq \delta) \\
& \stackrel{(a)}{\leq}\sum_{y \in \cup_{x \in \cS_{i-1}} \{x0, x1\}} \alpha \\
&  = 2 \alpha |\cS_{i-1}| 
\end{align*}
where $(a)$
follows from the assumption that $\Pr(|\hat P_y - P_y| \leq \delta) \geq 1- \alpha$ for all $y$.
 Hence, by the inductive hypothesis and the union bound, the probability that $|\cS_i|$ 
exceeds $ 1/\delta$ is at most $2(i-1)\alpha/\delta$.

Combining the above two results, we get that with probability $\leq 9(d-1)\alpha/(4\delta)$,
$\cS_d$ contains all symbols with probability $\geq 4\delta$ and
$|\cS_d| \leq 1/\delta$. 
Conditioned on this event, with probability at least $1-\alpha/\delta$, $|P_x
-\hat{P}_x|\leq \delta$, for all $x \in \cS_d$.  Hence,
\begin{align*}
\max_{x \in \{0,1\}^d} |P_x -\hat{P}_x| & = \max \left(\max_{x \in
  S_d} |P_x -\hat{P}_x| , \max_{x \notin S_d} P_x \right)
	\leq \max \left(\max_{x \in S_d} |P_x -\hat{P}_x|, 4\delta
\right) \leq 4\delta.
\end{align*}
By the union bound the total error probability is $9(d-1)\alpha/(4\delta) + \alpha /\delta$. 
\end{proof}

\section{Proofs of auxiliary results}
\label{app:aux}

\begin{proof}[Proof of \prettyref{lmm:aux-dual}]
\begin{enumerate}
	\item Notice that maximizer $\Delta$ for $\delta(t)$ yields a feasible solution $\lambda \Delta$ in the program
for $\delta(\lambda t)$.

\item The first inequality of~\eqref{eq:deltas} is obvious. Now consider $\Delta$ is the maximizer for $\delta(t)$. Let
$\iprod{\Delta}{\ones} = \epsilon$. 
Since $\Phi$ is column-stochastic, we have $\Phi^\top \ones = \ones$ and thus
$$\iprod{\Delta}{\ones} = \iprod{\Phi \Delta}{\ones} \le\|\Phi\Delta\|_1 \|\ones\|_\infty \le t $$ 
and thus $|\epsilon|\le t$.
Then, let $\pi_\pm$ be distributions on $\Theta$ with
$\iprod{\pi_-}{h}\le 0 \le \iprod{\pi_+}{h}$ (their existence follows from that of $\theta_{\pm}$). Let
$\pi_0=\pi_-$ if $\epsilon>0$ and $\pi_0=\pi_+$ otherwise. Define
$$ \tilde\Delta \eqdef {1\over 2} \Delta - {\epsilon\over 2} \pi_0\,.$$
From $\|\pi_-\|_1 = 1$, $\|\Phi r\|_1 \le \|r\|_1$ and triangle inequality we have:
\begin{align*} \|\tilde \Delta\| &\le {1\over 2} + {|\epsilon|\over 2} \le 1\\
   \|\Phi\tilde \Delta\| &\le {1\over 2} \|\Phi \Delta\|_1 + {|\epsilon|\over 2} \le t\\
   \Iprod{\tilde \Delta}{\ones} &= {1\over 2} \iprod{\Delta}{\ones} - {\epsilon\over 2}=0\\
   \Iprod{\tilde \Delta}{h} &= {1\over 2} \iprod{\Delta}{h} - {\epsilon\over 2} \iprod{\pi_0}{h} \ge {1\over 2}
   \iprod{\Delta}{h}\,.
\end{align*}
Thus, we get that $\tilde \delta(t) \ge {1\over 2} \delta(t)$. 

\item
Let $z$ be such that $\iprod{z}{\ones}=0$ and $\|z\|_1\leq 1$ and set $\Delta=t z$. Then 
	$\|\Phi\Delta\|_1 \leq \|\Phi\|_{\ell_1\to\ell_1} \|\Delta\|_1 \leq t$.
	Choose $z$ to maximize $\iprod{z}{h}$ gives the desired result with $C(h) = {h_{\max}-h_{\min}\over 2}$, where 
$h_{\max} > h_{\min}$ are the maximal and minimal values of $h(\theta)$.
\end{enumerate}
\end{proof}

\begin{proof}[Proof of \prettyref{lmm:gnorm}]
	Recall the following upper bound on the $A$-norm in terms of the $H^\infty$-norm of the derivative, which is a simple consequence of Parseval's identity and the Cauchy-Schwartz inequality (see, e.g., \cite{Newman75}):
\begin{equation}
\|g\|_{A} \leq \|g\|_{H^\infty(D)} + 2 \|g'\|_{H^\infty(D)}.
\label{eq:newman}
\end{equation}
Recall that $\|\delta^{1+z\over 1-z}\|_{H^\infty(D)} = 1$. We have $\|g\|_{H^\infty(D)} \leq 4$.
Furthermore, $g'(z) = 2 \delta^{1+z\over 1-z} (\log \delta + z - 1)$. We have $\|g'\|_{H^\infty(D)} \leq 2 (\log \frac{1}{\delta}+2)$ and hence the lemma.
\end{proof}

\begin{remark}[Alternative proof of \prettyref{lmm:gnorm}]
\label{rmk:lag}	
	In fact with more refined analysis it is possible to choose
	\[
g(z) = (1-z) \delta^{1+z\over 1-z}
\]
as opposed to \prettyref{eq:g_bsc}.
Of course, in this case the bound \prettyref{eq:newman} based on derivatives fails.
Instead, we can express the Taylor expansion coefficients of $g$ in terms of Laguerre polynomials. 
	Recall the generalized Laguerre polynomial $L_n^{(\nu)}(x)$ 
and its generating function \cite[8.975]{GR}:
\[
\sum_{n \geq 0} z^n L_n^{(\nu)}(x) =  \frac{1}{(1-z)^{1+\nu}}  \exp\pth{\frac{zx}{z-1}} , \quad |z|< 1.
\]
Since $g(z) = (1-z) \delta \delta^{\frac{2}{1-z}}$, its Taylor expansion $g(z) = \sum_{n \geq 0} a_n z^n$ is given by
$a_n = \delta L_n^{(-2)}(x)$ with $x = 2 \log \frac{1}{\delta}$ 
and hence $\|g\|_{A} = \sum_{n\geq 0} |L_n^{(-2)}(x)|$.
Recall the classical bound of Szeg\"o (cf.~\cite[(7.6.10)]{orthogonal.poly}) on Laguerre polynomials: for $\nu \leq -\frac{1}{2}$, as $n\to\infty$,
$|L_n^{(\nu)}(x)| = O(n^{\nu/2-1/4})$
uniformly in any compact interval of $[0,\infty)$.
Therefore $a_n = O(n^{-5/4})$ and hence $\|g\|_{A}$ is finite.
Furthermore, recall Fej\'er's sharp asymptotics of Laguerre polynomials \cite[Theorem 8.22.1]{orthogonal.poly}: for fixed $\nu \in \reals$ as $n\to\infty$, $L_n^{(\nu)}(x) = \sqrt{\pi} e^{x/2} x^{-\nu/2-1/4} n^{\nu/2-1/4} \cos(2\sqrt{nx} - \pi(1/4+\nu/2)) + O(n^{\nu/2-3/4})$, where $x>0$ and the remainder is uniform in any compact interval of $(0,\infty)$. This suggests that 
$\|g\|_{A} = O((\log \frac{1}{\delta})^{5/4})$.

%
%

\end{remark}

\begin{proof}[Proof of \prettyref{lmm:third}]
Note that $ \|\delta^{1+\alpha z\over 1-\alpha z} h(z)\|_{H^\infty(D)} = \max_{|z|=1} \big|\delta^{1+\alpha z\over 1-\alpha z} \frac{1}{(1- c (z-1))^d}\big|$, where
\begin{align*}
\left| \frac{\delta^{1+\alpha e^{i\theta}\over 1-\alpha e^{i\theta}}}{(1- c (e^{i\theta}-1))^d}\right|
= & ~ \delta^{\text{Re}(\frac{1+\alpha e^{i\theta}}{1-\alpha e^{i\theta}})} |1- c (e^{i\theta}-1)|^{-d} \\
= & ~ 	\delta^{\frac{1-\alpha^2}{1-2\alpha \cos \theta + \alpha^2}} (1 + 2c(1+c) (1-\cos \theta))^{-d/2}
= e^{-F(1-\cos \theta)}
\end{align*}
where $F(t) \triangleq \frac{1-\alpha^2}{(1-\alpha)^2+2\alpha t} \log \frac{1}{\delta}  +\frac{d}{2}\log(1 + 2 t \mu)$ and $\mu=\mu(\epsilon)$ as defined in \prettyref{eq:mu} since $c=\frac{\epsilon}{1-2\epsilon}$ and 
$c(1+c)=\mu$.
The goal is to identify conditions so that the $H^\infty$-norm is achieved at  $\theta= 0$ ($z=1$).
Let $s = 2t \mu, A = \frac{d}{2(1-\alpha^2) \log \frac{1}{\delta}}, B =(1-\alpha)^2,C=\alpha/\mu$. Then $F(t) = (1-\alpha^2) \log \frac{1}{\delta}(A \log(1+s) + \frac{1}{B+Cs}) \triangleq \delta(s)$.
Straightforward calculation shows that 
$F'(t) \geq 0$ for $t\in[0,2]$ provided that
\[
A \geq \frac{C}{B^2},   \quad A \geq \frac{1}{2B},
\]
that is,
$d \geq \log \frac{1}{\delta} \max\{ \frac{\alpha(1-\alpha^2)}{\mu (1-\alpha)^4}, \frac{1-\alpha^2}{(1-\alpha)^2}\}$, which is ensured by \prettyref{eq:d1}.
Hence $\|\delta^{1+\alpha z\over 1-\alpha z} h(z)\|_{H^\infty(D)}=e^{-F(0)}$ as claimed.
\end{proof}

Finally, the next result is used in \prettyref{sec:lp-bec}:
\begin{lemma}
\label{lmm:bin}	
	For any $d'\geq d \geq 1$ and $0\leq p<1$,
	$H^2(\Binom(d',p), \Binom(d'-d,p)) \leq \frac{4pd^2}{(1-p) d'}$. 
\end{lemma}
\begin{proof}[Proof of \prettyref{lmm:bin}]
	It suffices to show that for any $n\geq 1$, $H^2(\Binom(n,p),\Binom(n-1,p)) \leq \frac{p}{(1-p) n}$.	
	Indeed, since $H^2(P,Q) \leq \chi^2(P\|Q) = \int \frac{(dP)^2}{dQ} -1$, we have
	\[
	\chi^2(\Binom(n-1,p)\|\Binom(n,p)) = \Expect_{X \sim \Binom(n,p)}\Big[\Big(\frac{n-X}{n(1-p)}\Big)^2\indc{X < n} \Big] \leq \frac{p}{n(1-p)}.
		\]
	Then by the triangle inequality of the Hellinger distance, we have
	$H(\Binom(d',p), \Binom(d'-d,p)) \leq \sqrt{\frac{p}{1-p}} \sum_{n=d'-d+1}^{d'} \frac{1}{\sqrt{n}} \leq \sqrt{\frac{p}{1-p}} \int_{d'-d}^{d'} \frac{1}{\sqrt{x}}dx \leq \sqrt{\frac{p}{1-p}} \frac{2d}{\sqrt{d'}} $. 
\end{proof}

\section*{Acknowledgement}
The research was supported (in part) by the Center for Science of Information (CSoI),
an NSF Science and Technology Center, under grant agreement CCF-09-39370, 
ITS-1447879, CCF-1527105, NSF CAREER awards CCF-12-53205 and CCF-1651588.


\begin{thebibliography}{26}
\providecommand{\natexlab}[1]{#1}
\providecommand{\url}[1]{\texttt{#1}}
\expandafter\ifx\csname urlstyle\endcsname\relax
  \providecommand{\doi}[1]{doi: #1}\else
  \providecommand{\doi}{doi: \begingroup \urlstyle{rm}\Url}\fi

\bibitem[Abramowitz and Stegun(1964)]{AS64}
M.~Abramowitz and I.~A. Stegun.
\newblock \emph{{Handbook of mathematical functions with formulas, graphs, and
  mathematical tables}}.
\newblock Wiley-Interscience, New York, NY, 1964.

\bibitem[Batman et~al.(2013)Batman, Impagliazzo, Murray, and
  Paturi]{batman2013finding}
Lucia Batman, Russell Impagliazzo, Cody Murray, and Ramamohan Paturi.
\newblock Finding heavy hitters from lossy or noisy data.
\newblock In \emph{Approximation, Randomization, and Combinatorial
  Optimization. Algorithms and Techniques}, pages 347--362. Springer, 2013.

\bibitem[Borwein and Erd\'elyi(1997)]{borwein1997littlewood}
Peter Borwein and Tam\'as Erd\'elyi.
\newblock Littlewood-type problems on subarcs of the unit circle.
\newblock \emph{Indiana University Mathematics Journal}, 46\penalty0
  (4):\penalty0 1323, 1997.

\bibitem[Cai and Low(2005)]{cai2005nonquadratic}
T~Tony Cai and Mark~G Low.
\newblock Nonquadratic estimators of a quadratic functional.
\newblock \emph{The Annals of Statistics}, pages 2930--2956, 2005.

\bibitem[Cowen(1988)]{cowen1988linear}
Carl~C Cowen.
\newblock Linear fractional composition operators on {$H^2$}.
\newblock \emph{Integral equations and operator theory}, 11\penalty0
  (2):\penalty0 151--160, 1988.

\bibitem[De et~al.(2016{\natexlab{a}})De, O'Donnell, and
  Servedio]{de2016optimal}
Anindya De, Ryan O'Donnell, and Rocco Servedio.
\newblock Optimal mean-based algorithms for trace reconstruction.
\newblock \emph{arXiv preprint arXiv:1612.03148}, 2016{\natexlab{a}}.

\bibitem[De et~al.(2016{\natexlab{b}})De, Saks, and Tang]{DST16}
Anindya De, Michael Saks, and Sijian Tang.
\newblock Noisy population recovery in polynomial time.
\newblock \emph{arXiv preprint arXiv:1602.07616}, 2016{\natexlab{b}}.

\bibitem[De et~al.(2017)De, O'Donnell, and Servedio]{de2017sharp}
Anindya De, Ryan O'Donnell, and Rocco Servedio.
\newblock Sharp bounds for population recovery.
\newblock \emph{arXiv preprint arXiv:1703.01474}, 2017.

\bibitem[Dvir et~al.(2012)Dvir, Rao, Wigderson, and Yehudayoff]{DRWY12}
Zeev Dvir, Anup Rao, Avi Wigderson, and Amir Yehudayoff.
\newblock Restriction access.
\newblock In \emph{Proceedings of the 3rd Innovations in Theoretical Computer
  Science Conference}, pages 19--33. ACM, 2012.

\bibitem[Fan et~al.(2015)Fan, Rigollet, and Wang]{FRW15}
Jianqing Fan, Philippe Rigollet, and Weichen Wang.
\newblock Estimation of functionals of sparse covariance matrices.
\newblock \emph{The Annals of Statistics}, 43\penalty0 (6):\penalty0 2706,
  2015.

\bibitem[Feldman et~al.(2008)Feldman, O'Donnell, and
  Servedio]{feldman2008learning}
Jon Feldman, Ryan O'Donnell, and Rocco~A Servedio.
\newblock Learning mixtures of product distributions over discrete domains.
\newblock \emph{SIAM Journal on Computing}, 37\penalty0 (5):\penalty0
  1536--1564, 2008.

\bibitem[Gradshteyn and Ryzhik(2007)]{GR}
I.~S. Gradshteyn and I.~M. Ryzhik.
\newblock \emph{{Table of Integrals Series and Products}}.
\newblock Academic Press, New York, NY, seventh edition, 2007.

\bibitem[Ibragimov et~al.(1987)Ibragimov, Nemirovskii, and Khas'minskii]{INK87}
I.A. Ibragimov, A.S. Nemirovskii, and R.Z. Khas'minskii.
\newblock Some problems on nonparametric estimation in gaussian white noise.
\newblock \emph{Theory of Probability \& Its Applications}, 31\penalty0
  (3):\penalty0 391--406, 1987.

\bibitem[Kearns et~al.(1994)Kearns, Mansour, Ron, Rubinfeld, Schapire, and
  Sellie]{kearns1994learnability}
Michael Kearns, Yishay Mansour, Dana Ron, Ronitt Rubinfeld, Robert~E Schapire,
  and Linda Sellie.
\newblock On the learnability of discrete distributions.
\newblock In \emph{Proceedings of the twenty-sixth annual ACM symposium on
  Theory of computing}, pages 273--282. ACM, 1994.

\bibitem[Le~Cam(1986)]{Lecam86}
Lucien Le~Cam.
\newblock \emph{Asymptotic methods in statistical decision theory}.
\newblock Springer-Verlag, New York, NY, 1986.

\bibitem[Li et~al.(2015)Li, Rabani, Schulman, and Swamy]{li2015learning}
Jian Li, Yuval Rabani, Leonard~J Schulman, and Chaitanya Swamy.
\newblock Learning arbitrary statistical mixtures of discrete distributions.
\newblock In \emph{Proceedings of the Forty-Seventh Annual ACM on Symposium on
  Theory of Computing}, pages 743--752. ACM, 2015.

\bibitem[Lounici(2014)]{Lounici14}
Karim Lounici.
\newblock High-dimensional covariance matrix estimation with missing
  observations.
\newblock \emph{Bernoulli}, 20\penalty0 (3):\penalty0 1029--1058, 2014.

\bibitem[Lovett and Zhang(2015)]{LZ15}
Shachar Lovett and Jiapeng Zhang.
\newblock Improved noisy population recovery, and reverse {B}onami-{B}eckner
  inequality for sparse functions.
\newblock In \emph{Proceedings of the Forty-Seventh Annual ACM on Symposium on
  Theory of Computing}, pages 137--142. ACM, 2015.

\bibitem[Moitra and Saks(2013)]{MS13}
Ankur Moitra and Michael Saks.
\newblock A polynomial time algorithm for lossy population recovery.
\newblock In \emph{Foundations of Computer Science (FOCS), 2013 IEEE 54th
  Annual Symposium on}, pages 110--116. IEEE, 2013.

\bibitem[Nazarov and Peres(2016)]{nazarov2016trace}
F.~Nazarov and Y.~Peres.
\newblock Trace reconstruction with $\exp({O}(n^{1/3}))$ samples.
\newblock \emph{arXiv preprint arXiv:1612.03599}, 2016.

\bibitem[Newman(1975)]{Newman75}
DJ~Newman.
\newblock A simple proof of {W}iener’s $1/f$ theorem.
\newblock \emph{Proceedings of the American Mathematical Society}, 48\penalty0
  (1):\penalty0 264--265, 1975.

\bibitem[Orlitsky et~al.(2016)Orlitsky, Suresh, and Wu]{OSW16}
Alon Orlitsky, Ananda~Theertha Suresh, and Yihong Wu.
\newblock Optimal prediction of the number of unseen species.
\newblock \emph{Proceedings of the National Academy of Sciences (PNAS)},
  113\penalty0 (47):\penalty0 13283--13288, 2016.

\bibitem[Simon(2011)]{simon2011convexity}
Barry Simon.
\newblock \emph{Convexity: An analytic viewpoint}.
\newblock Cambridge University Press, 2011.

\bibitem[Szeg{\"o}(1975)]{orthogonal.poly}
G.~Szeg{\"o}.
\newblock \emph{{Orthogonal polynomials}}.
\newblock American Mathematical Society, Providence, RI, 4th edition, 1975.

\bibitem[Tsybakov(2009)]{Tsybakov09}
A.~B. Tsybakov.
\newblock \emph{Introduction to Nonparametric Estimation}.
\newblock Springer Verlag, New York, NY, 2009.

\bibitem[Wigderson and Yehudayoff(2012)]{WY12}
Avi Wigderson and Amir Yehudayoff.
\newblock Population recovery and partial identification.
\newblock In \emph{Foundations of Computer Science (FOCS), 2012 IEEE 53rd
  Annual Symposium on}, pages 390--399. IEEE, 2012.

\end{thebibliography}

\end{document}